\documentclass{amsart}
\usepackage[margin=3.5cm]{geometry}
 
\DeclareMathAccent{\vec}{\mathord}{letters}{"7E}
\usepackage{amsmath}  
\usepackage{graphicx} 
\usepackage{epstopdf} 
\usepackage{amssymb} 
\usepackage{mathtools}
\usepackage{amstext}
\usepackage{relsize}
\usepackage{fancyhdr}
\usepackage{array}
\usepackage{float}
\usepackage{graphicx,xcolor,colortbl}    
\usepackage{hyperref}  
\usepackage{cleveref}
\usepackage{xspace}
\usepackage[title]{appendix}

\newtheorem{assumption}{Assumption}[section] 
\newtheorem{lemma}{Lemma}[section]
\newtheorem{definition}{Definition}[section]
\newtheorem{remark}{Remark}[section]
\newtheorem{theorem}{Theorem}[section]

\newcommand\numberthis{\addtocounter{equation}{1}\tag{\theequation}}

\newcommand{\hmin}{h_{\min}}
\newcommand{\hmax}{h_{\max}}
\newcommand{\ito}{It\^{o}\xspace}
\newcommand{\Cx}{C_{\texttt{X}}\xspace}

\newcommand{\prob}[1]{\mathbb{P}\left[  {#1} \right]}
\newcommand{\expect}[1]{\mathbb{E}\left[  {#1} \right]}
\newcommand{\Cp}[1]{\mathbf{\gamma}_{#1}}

\newcommand{\Clevy}[2][]{C^{#1}_{\texttt{LA}}\left(#2\right)} 
\newcommand{\Cscheme}[2][]{C^{#1}_{\texttt{SR}}\left(#2\right)} 

\newcommand{\CpathBar}[3][]{\overline{C}^{#1\{#2\}}_{\texttt{PR}}}

\newcommand{\Cpath}[2][]{C^{#1}_{\texttt{PR}}\left(#2\right)} 
\newcommand{\Citerated}[2][]{C^{#1}_{\texttt{ISI}}\left(#2\right)} 
\newcommand{\Ctaylor}[2][]{C^{#1}_{\texttt{TE}}\left(#2\right)} 

\newcommand{\NOTE}[1]{\textcolor{black}{{#1}}}
\newcommand{\Note}[1]{\textcolor{black}{{#1}}} 
 
\newcommand{\NoteTypo}[1]{\textcolor{black}{{#1}}} 

\begin{document} 




\title[An adaptive time-stepping Milstein method]{Strong convergence of an adaptive time-stepping Milstein method for SDEs with monotone coefficients}

\author{C\'onall Kelly}
\address{School of Mathematical Sciences, University College Cork, Ireland.}
\curraddr{}
\email{conall.kelly@ucc.ie}
\thanks{}

\author{Gabriel Lord}
\address{Mathematics, IMAPP, Radboud University, Nijmegen, The Netherlands.}
\email{gabriel.lord@ru.nl}
\thanks{}

\author{Fandi Sun}
\address{Maxwell Institute, Department of Mathematics, MACS, Heriot-Watt University, Edinburgh, UK.}
\curraddr{}
\email{fs30@hw.ac.uk}
\thanks{}

\date{\today}


\begin{abstract} 
We introduce an explicit adaptive Milstein method for stochastic differential equations (SDEs) with no commutativity condition. The drift and diffusion are separately locally Lipschitz and together satisfy a monotone condition. This method relies on a class of path-bounded time-stepping strategies which work by reducing the stepsize as solutions approach the boundary of a sphere, invoking a backstop method in the event that the timestep becomes too small. We prove that such schemes are strongly $L_2$ convergent of order one.  This order is inherited by an explicit adaptive Euler-Maruyama scheme in the additive noise case. Moreover we show that the probability of using the backstop method at any step can be made arbitrarily small. We compare our method to other fixed-step Milstein variants on a range of test problems.
\end{abstract} 

\maketitle

\section{Introduction} \label{sec:intro}
We investigate the use of adaptive time-stepping strategies in the construction of a strongly convergent explicit Milstein-type numerical scheme for a \NOTE{$d$}-dimensional stochastic differential equation (SDE) of \ito-type
\NOTE{on the probability space $(\Omega , \mathcal{F} , \mathbb{P})$,}
\begin{align}
X(t)= X(0)+\int_{0}^{t}f(X(r))dr +\sum_{i=1}^{m}\int_{0}^{t}g_i(X(r))dW_i(r)\label{eq:true},\end{align}
for $t\in[0,T]$, $T\geq 0$ and $i=1,\dots, m\in\mathbb{N}$, where $W=[W_1,\cdots, W_m]^T$ is an $m$-dimensional Wiener process, the drift coefficient $f: \mathbb{R}^d \rightarrow \mathbb{R}^d$ and the diffusion coefficient $g: \mathbb{R}^d\rightarrow \mathbb{R}^{d\times m}$ each satisfy a local Lipschitz condition along with a polynomial growth condition and, together, a monotone condition. Both are twice continuously differentiable; see Assumption \ref{ass:f+g} and Assumption \ref{ass:SDEmoments_power}. Throughout, we take the initial vector $X(0)=X_0\in\mathbb{R}^d$ to be deterministic.

It was pointed out in \cite{wang2013tamed} that, because the Euler-Maruyama and Euler-Milstein methods coincide in the additive noise case, and as a consequence of the analysis in \cite{hutzenthaler2011strong}, an explicit Milstein scheme over a uniform mesh cannot converge in ${L}_p$ to solutions of \eqref{eq:true}. 
We propose here an adaptive variant of the explicit Milstein method that achieves strong ${L}_2$ convergence of order one to solutions of \eqref{eq:true}. As an immediate consequence of this, in the case of additive noise an adaptive Euler-Maruyama method also has ${L}_2$ convergence  of order one. To prove our convergence result it is essential to introduce a new variant of the admissible class of time-stepping strategies introduced in \cite{kelly2021adaptive,kelly2018adaptive}, which we call path-bounded strategies.

Several variants on the fixed-step Milstein method have been proposed, see for example the tamed Milstein~\cite{wang2013tamed,KumarSabanis2019}, projected and split-step backward Milstein~\cite{beyn2017stochastic}, truncated Milstein~\cite{Guo2018}, implicit Milstein methods~\cite{HighamMaoSzpruch2013,YaoGan18} and a recent tamed stochastic Runge-Kutta (of order one) method of \cite{gan2020tamed}, all designed to converge strongly to solutions of SDEs with more general drift and diffusions, such as in \eqref{eq:true}. However, with few exceptions (see \cite{KumarSabanis2019,beyn2017stochastic}) explicit methods of this kind have only examined the case where the diffusion coefficients $g_i$ satisfies a commutativity condition. We do not impose a commutativity restriction and hence must consider the associated L\'evy areas (see Lemma \ref{lem:levy bound}).

A review of methods that adapt the timestep in order to control local error may be found in the introduction to \cite{kelly2018adaptive}; we cite here ~\cite{Burrage2004,LMS2006,MR3325826,shardlow2016pathwise,Gaines1997,Mauthner1998} and remark that our purpose is instead to handle the nonlinear response of the discrete system see also \cite{fang2016adaptive,fang2020adaptive} and discussion in \cite{kelly2018adaptive,kelly2021adaptive}.
A common feature of the adaptivity is the use of both a minimum and maximum time step where the magnitude of the minimum step is controlled by a free parameter which requires some a-priori knowledge on the part of the user. The approach of 
\cite{fang2016adaptive,fang2020adaptive} was recently extended to McKean-Vlasov equations in \cite{reisinger2020adaptive} and include a Milstein approximation. 
In addition we note the fully adaptive Milstein method proposed in \cite{Hofmann2001} for a scalar SDE with light constraints on the coefficients. There the authors stated that such a method was easy to implement but hard to analyse and as a result 
considered a different, but related method.

Our framework for adaptivity was introduced in \cite{kelly2018adaptive} for an explicit Euler-Maruyama method, and has since been extended to SDE systems with monotone coefficients in \cite{kelly2021adaptive} and to SPDE methods in \cite{campbell2018adaptive}. 
These methods all use a backstop method when the chosen strategy attempts to select a stepsize below the minimum step. 
We demonstrate here, for a path-bounded strategy, that the probability of using the backstop method can be made arbitrarily small by choosing an appropriately large $\rho$, and an appropriately small $h_{\max}$. This is consistent with observation, and with the intuitive notion that the use of the backstop method should be rare in practice (see (e) and (f) of Figure \ref{fig:scalarConv}).

The structure of the article is as follows. Mathematical preliminaries are considered in Section \ref{sec:prelim}, including precise specifications of the conditions imposed on each $f$ and $g_i$, and the characterisation of an explicit Milstein method on an arbitrary mesh. The construction and result of the adaptive time-stepping strategy is outlined in Section \ref{sec:adap}, where we formulate the adaptive Milstein scheme with backstop which will be the subject of our main theorem. Both main results: on strong $L_2$ convergence and on the probability of using the backstop method, are stated in Section \ref{subsec:converg}; we defer their proofs to Section \ref{sec:proof}. 
In Section \ref{sec:num} we compare the adaptive scheme numerically to other fixed step methods and illustrate both convergence and efficiency. The proof of Lemma \ref{lem:levy bound} is in Appendix \ref{sec:appendix}.

\section{Mathematical preliminaries} \label{sec:prelim}

We consider the $d$-dimensional \ito-type SDE \eqref{eq:true} and for the remainder of the article let $(\mathcal{F}_t)_{t\geq 0}$ be the natural filtration of $W$. For all $x \in \mathbb{R}^d$ and for all $\phi(x)\in \mathrm{C}^2(\mathbb{R}^d,\mathbb{R}^d)$, the Jacobian matrix of $\phi(x)$ is denoted $\mathbf{D}\phi(x)\in\mathcal{L}(\mathbb{R}^d,\mathbb{R}^d)$; the second derivative of $\phi(x)$ with respect to a vector $x$ forms a 3-tensor and is denoted $\mathbf{D}^2\phi(x)\in\mathcal{L}(\mathbb{R}^{d\times d},\mathbb{R}^d)$; and $[x]^2:=x\otimes x$ stands for the outer product of $x$ and itself.
Furthermore, let $\|\cdot\|$ denote the standard $l^2$ norm in $\mathbb{R}^d$, $\|\cdot\|_{\mathbf{F}(a\times b)}$ the Frobenious norm of the matrix in $\mathbb{R}^{a\times b}$; for simplicity we write $\|\cdot\|_{\mathbf{F}}$ as the Frobenious norm of the matrix in $\mathbb{R}^{d\times d}$. $\|\cdot\|_{\mathbf{T}_3}$ denotes the induced tensor norm (spectral norm) of the 3-tensor in $\mathbb{R}^{d\times d\times d}$ and it is defined as $\big\|\cdot \big\|_{\mathbf{T}_3}:=\sup_{h_1,h_2\in\mathbb{R}^d, \|h_1\|,\|h_2\|\leq 1}\big\|\cdot(h_1\otimes h_2)\big\|$.
For $a,b\in\mathbb{R}$, $a\vee b$ denotes max$\{a,b\}$ and $a\wedge b$ denotes min$\{a,b\}$. 
We frequently make use of the elementary inequality 
\begin{align}
2ab\leq a^2+b^2,\quad a,b\in\mathbb{R}, \label{eq:ab<a^2+b^2}
\end{align} 
and of the following two standard extensions of Jensen's inequality 
(see \cite[Corollary A.10]{lord2014introduction}).
For $f\in L^1$, if $p\geq 1$, 
\begin{align}
    \Bigg| \int_{0}^t f(s)ds \Bigg|^p \leq t^{p-1}\int_{0}^t |f(s)|^p ds,\quad  t\geq 0. \label{eq:Jen_int}
\end{align}
For $a_i\in\mathbb{R}$ and $p\geq 1$,
\begin{align}
    \Bigg| \sum_{i=1}^{n} a_i \Bigg|^p \leq n^{p-1}\sum_{i=1}^{n} |a_i|^p, \quad n\in\mathbb{N}\backslash \{ 0\}. \label{eq:Jen_sum}
\end{align}
We now present our assumptions on $f$ and $g_i$ in \eqref{eq:true}.
\begin{assumption} \label{ass:f+g}
Let drift $f(x)\in \mathrm{C}^2(\mathbb{R}^d,\mathbb{R}^d)$ and diffusion $g(x)\in\mathrm{C}^2(\mathbb{R}^d,\mathbb{R}^{d\times m})$ with its $i$-th column 
$g_i(x)=[g_{1,i}(x),\dots,g_{d,i}(x)]^T\in\mathrm{C}^2(\mathbb{R}^d,\mathbb{R}^{d})$ for $i=1,\dots,m$. For each $\varkappa \geq 1$  there exist $L_{\varkappa}>0$ such that
\begin{align}
    \big\|f(x)-f(y)\big\|^2+\big\|g(x)-g(y)\big\|^2_{\mathbf{F}(d\times m)}\leq L_{\varkappa}\big\|x-y\big\|^2,  \label{eq:local_lipschitz}
\end{align}
for $x,y\in\mathbb{R}^d$ with $\|x\| \vee \|y\|\leq \varkappa$, and there exists $c\geq 0$ such that for some \Note{$\eta \geq 2$}
\begin{align}
    \big\langle x-y, f(x)-f(y)\big\rangle+\frac{\eta-1}{2} \big\|g(x)-g(y)\big\|^2_{\mathbf{F}(d\times m)} \leq c\big\|x-y\big\|^2.  \label{eq:monotone}
\end{align}
In addition, for some constants $c_{  3,4,5,6  }$, $q_1$, $q_2\geq 0$; $i=1,\dots,m$, we have
\begin{alignat}{2}
   \big \|\mathbf{D}f(x)\big\|_{\mathbf{F}}\leq\,\,& c_3(1+\|x\|^{q_1+1}), \qquad\quad \big\|\mathbf{D}g_i(x)\big\|_{\mathbf{F}}&&\leq\,\, c_4(1+\|x\|^{q_2+1}), \label{eq:Df+Dg}\\
    \big\| f(x)\big\| \leq\,\,& c_5(1+\|x\|^{q_1+2}), \quad\quad \big\| g(x)\big\|_{\mathbf{F}(d\times m)} &&\leq\,\, c_6(1+\|x\|^{q_2+2}).\label{eq:||f||+||g||}
\end{alignat}
Furthermore, for some $c_{1,2}\geq 0$; $i=1,\dots,m$, we have
\begin{align}
    \big\|\mathbf{D}^2f(x)\big\|_{\mathbf{T}_3}\leq\,\,
    c_1(1+\|x\|^{q_1}), \quad \big\|\mathbf{D}^2g_i(x)\big\|_{\mathbf{T}_3}\leq\,\, c_2(1+\|x\|^{q_2}).\label{eq:D^2f+D^2g}
\end{align}
\end{assumption}
Under \eqref{eq:local_lipschitz} and \eqref{eq:monotone}, the SDE \eqref{eq:true} has a unique strong solution on any interval $[0, T ]$, where $T < \infty$ on the filtered probability space $(\Omega , \mathcal{F} , (\mathcal{F}_t )_{t \geq 0} , \mathbb{P})$, see~\cite{Hasminskii}, \cite{mao2007SDEapp} and \cite{tretyakov2013fundamental}. 
\begin{assumption} \label{ass:SDEmoments_power}
Suppose that \eqref{eq:monotone} in Assumption \ref{ass:f+g} holds with 
\begin{align*}
    \eta \geq 4q  + 2q_2+10,
\end{align*}
where $q:=q_1\vee q_2$, $q_1$ and $q_2$ are from \eqref{eq:||f||+||g||} in Assumption \ref{ass:f+g}.
\end{assumption}
We now give the following Lemma on moments of the solution.
\begin{lemma}{\cite[Lem. 4.2]{mao2015truncated}}\label{lem:boundedMomentsSDE}
Let $f$ and $g$ satisfy \eqref{eq:local_lipschitz}, and suppose that Assumption \ref{ass:SDEmoments_power} holds.  If $g$ further satisfies \eqref{eq:||f||+||g||}, then there is a constant $\Cx>0$ such that the solution of \eqref{eq:true} satisfies
\begin{equation}\label{eq:SDEmoments}
\mathbb{E}\biggl[\sup_{s\in[0,T]}\|X(s)\|^{\eta-2q_2-2}\biggr] \leq \Cx.
\end{equation}
\end{lemma}

Next we present the fixed-step Milstein method (see \cite[Sec. 10.3]{kloeden1991numerical}) that is the basis of the adaptive method presented in this article.
\begin{definition}[Milstein method]
\label{def:ExplicitMilstein}
For $n\in\mathbb{N}$, $s\in[t_n, t_{n+1}]$ and given $Y(t_n)$, the fixed-step Milstein scheme for \eqref{eq:true}, interpolated over the interval $[t_n,t_{n+1}]$, is given by
\begin{multline}\label{eq:ExplicitMilstein-integral}
\qquad Y(s):=Y(t_n)+f\big(Y(t_n)\big)|s-t_n|+\sum_{i=1}^{m}g_i\big(Y(t_n)\big)I_{i}^{t_n,s}\\
+\sum_{i,j=1}^{m}\mathbf{D}g_i\big(Y(t_n)\big)g_j\big(Y(t_n)\big)I_{j,i}^{t_n,s},  
\end{multline}
where following \cite{wang2013tamed,beyn2017stochastic}, the stochastic integral and the iterated stochastic integral are defined as 
\begin{align}
    I_{i}^{t_n,s}:=\int_{t_{n}}^{s}dW_i(r), \qquad I_{j,i}^{t_n,s}:=\int_{t_{n}}^{s} \int_{t_{n}}^{r}dW_j(p) dW_i(r). \label{eq:defISI}
\end{align}
\end{definition}
Expanding the last term in \eqref{eq:ExplicitMilstein-integral} we have that 
\begin{align*}
    &\sum_{i,j=1}^{m}\mathbf{D}g_i\big(Y(t_n)\big)g_j\big(Y(t_n)\big)I_{j,i}^{t_n,s}\\
    =&\frac{1}{2}\sum_{i=1}^{m}\mathbf{D}g_i\big(Y(t_n)\big)g_i\big(Y(t_n)\big)\left(\left(I_{i}^{t_n,s}\right)^2-|s-t_n|\right)\\
    &+\frac{1}{2}\sum_{\substack{i,j=1\\i<j}}^{m}\Big(\mathbf{D}g_i\big(Y(t_n)\big)g_j\big(Y(t_n)\big)+\mathbf{D}g_j\big(Y(t_n)\big)g_i\big(Y(t_n)\big)\Big)I_{i}^{t_n,s}I_{j}^{t_n,s}\\
    &+\sum_{\substack{i,j=1\\i<j}}^{m}\Big(\mathbf{D}g_i\big(Y(t_n)\big)g_j\big(Y(t_n)\big)-\mathbf{D}g_j\big(Y(t_n)\big)g_i\big(Y(t_n)\big)\Big)A_{ij}^{t_n,s},  \numberthis \label{eq:Dgexpansion}
\end{align*}
where the term $A_{ij}^{t_n,s}$ is the L\'evy area (see for example \cite[Eq. (1.2.2)]{levy1951wiener}) defined by
\begin{align}
A_{ij}^{t_n,s}:=\frac{1}{2}\left(I_{i,j}^{t_n,s}-I_{j,i}^{t_n,s}\right)\NOTE{,}
\label{def:levy area} 
\end{align}  
\NOTE{and we have used the relations $I_{i,i}^{t_n,s} = \frac{1}{2}( (I_{i}^{t_n,s})^2 - |t-s|)$
and
$I_{i,j}^{t_n,s} + I_{j,i}^{t_n,s} = I_{i}^{t_n,s} I_{j}^{t_n,s}$.} 
As mentioned in the introduction many authors assume the following commutativity condition: suppose that $\mathbf{D}g_i(y)g_j(y)=\mathbf{D}g_j(y)g_i(y)$ for all $i,j=1,\dots, m$ and $y\in\mathbb{R}^d$.
When this holds, the last term
in \eqref{eq:Dgexpansion} vanishes, avoiding the need for any analysis of $A_{ij}^{t_n,s}$ defined in \eqref{def:levy area}. We do not impose such a condition 
in this paper, and therefore make use of the following conditional moment bounds on the L\'evy areas.
\begin{lemma}[\texttt{L\'evy Area}]\label{lem:levy bound}
For all $i,j =1,\dots,m$, $0\leq t_n\leq s<T$ and for a pair of Wiener process $(W_i(r),W_j(r))^T$ where $r\in[t_n,s]$ and the L\'evy area $A_{ij}^{t_n,s}$  defined in \eqref{def:levy area}, there exists a finite constant 
$C_{\texttt{LA}}$ whose explicit form is in \eqref{eq:widehat_Ib} such that for $k\geq 1$ 
\begin{align}
\mathbb{E}\left[\big|A_{ij}^{t_n,s}\big|^k \middle|\mathcal{F}_{t_n}\right]\leq  \Clevy{k}\,|s-t_n|^k \quad a.s.\label{eq:levy moment}
\end{align}
\end{lemma}
For proof see Appendix \ref{sec:levybound}.
\section{Adaptive time-stepping strategies} \label{sec:adap}


To deal with the extra terms that arise from Milstein over Euler-Maruyama type discretisations, we introduce a new class of time-stepping strategies in Definition \ref{def:defYn}. Let $\{h_{n+1}\}_{n\in\mathbb{N}}$ be a sequence of strictly positive random timesteps with corresponding random times $\{t_n:=\sum_{i=1}^{n}h_i\}_{n\in\mathbb{N}\backslash \{0\}}$, where $t_0=0$.  
\begin{definition}\label{def:filtration}
Suppose that each member of $\{t_n\}_{n\in\mathbb{N}\backslash \{0\}}$ is an $\mathcal{F}_t$-stopping time: i.e. $\{t_n\leq t\}\in\mathcal{F}_t$ for all $t\geq 0$, where $(\mathcal{F}_t)_{t\geq 0}$ is the natural filtration of $W$. 
\Note{If $\tau$ is any $(\mathcal{F}_t)$-stopping time, then
(see \cite[p. 14]{mao2006stochastic}) }
\begin{equation}\label{eq:filtration}
\mathcal{F}_{\tau}:=\{A\in\mathcal{F}\,:\,A\cap\{\tau\leq t\}\in\mathcal{F}_t \Note{,\,\text{ for all }\,t\geq 0}\}.
\end{equation}
In particular this allows us to condition on $\mathcal{F}_{t_n}$ at any point on the random time-set $\{t_n\}_{n\in\mathbb{N}}$.
\end{definition}
\begin{assumption} \label{ass:h}
\Note{For the sequence of random timesteps $\{h_{n+1}\}_{n\in\mathbb{N}}$, there are constant values $\hmax>h_{\min}>0$, $\rho>1$ such that $h_{\max}=\rho  h_{\min}$, and }
\begin{equation}\label{eq:hmaxhmin}
0<h_{\min} \leq h_{n+1} \leq h_{\max}\leq 1.
\end{equation}
In addition, we assume each $h_{n+1}$ is $\mathcal{F}_{t_n}$-measurable. 
\end{assumption}
\begin{definition}\label{def:N}
Let $N^{(t)}$ be a random integer such that 
 \begin{align}
 N^{(t)}:=\max\{n\in\mathbb{N}\backslash \{ 0\}: t_{n-1}<\Note{t} \},   \label{eq:def_N}
 \end{align}
and let $N=N^{(T)}$ and $t_N=T$, so that $T$ is always the last point on the mesh. Note that $N^{(t)}$ indicates the step number such that $t\in \big[t_{N^{(t)}-1},\,t_{N^{(t)}}\big]$.
Furthermore, by Assumption \ref{ass:h}, $N^{(t)}$ only takes values in the finite set $\{N^{(t)}_{\min},\dots,N^{(t)}_{\max}\}$, where $N^{(t)}_{\min}:=\lfloor t/\hmax \rfloor$ and $N^{(t)}_{\max}:=\lceil t/\hmin \rceil$.
\end{definition}
In Assumption \ref{ass:h}, the lower bound $\hmin$ given by \eqref{eq:hmaxhmin} ensures that a simulation over the interval $[0,T]$ can be completed in a finite number of time steps. In the event that at time $t_n$ our strategy attempts to select a stepsize $h_{n+1} \leq \hmin$, we instead apply a single step of a backstop method ($\varphi$ in Definition \ref{def:adaptive explicit Milstein scheme} below), a known method that satisfies a mean-square consistency requirement with deterministic step $h_{n+1}=\hmin$ (see also discussion in Remarks \ref{rem:com1} \Note{and \ref{rem:lastStep}}).

First we recall the Milstein method expressed as a map. Over each step $[t_n,t_{n+1}]$ the Milstein map $\theta:\mathbb{R}^d\times \mathbb{R}  \times\mathbb{R}\rightarrow \mathbb{R}^d$ is defined as
\begin{align}
\theta\big(x,t_n,s-t_n\big):= x+(s-t_n)f(x) +\sum_{i=1}^{m}g_i(x)I_{i}^{t_n,s}+\sum_{i,j=1}^{m}\mathbf{D}g_i(x)g_j(x) I_{j,i}^{t_n,s}.  \label{eq:deftheta}
\end{align}

Following \cite[Def. 9]{kelly2021adaptive}, we  now define an adaptive Milstein scheme combining the Milstein method and a backstop method. 
\begin{definition}[Adaptive Milstein Scheme] \label{def:adaptive explicit Milstein scheme}
Let  $\{h_{n+1}\}_{n\in \mathbb{N}}$ satisfy Assumption \ref{ass:h}. Using indicator functions to distinguish the backstop case when $h_{n+1}=\hmin$ (and allowing for the possibility that the final step taken to time $T$ is smaller than $h_{\min}$, in which case the backstop is also used), we define the continuous form of an \textit{adaptive Milstein scheme} \Note{associated with a particular time-stepping strategy $\{h_{n+1}\}_{n\in\mathbb{N}}$} 
as 
\begin{multline}
\qquad\widetilde{Y}(s):=\theta\left(\widetilde Y(t_n)\boldsymbol{,}\,\,  t_n\boldsymbol{,}\,\,s-t_n\right) \cdot \mathbf{1}_{\{h_{\min}<h_{n+1}\leq h_{\max}\}}\\
+\varphi\left(\widetilde Y(t_n)\boldsymbol{,}\,\,t_n\boldsymbol{,}\,\,\Note{s-t_n}\right)  \cdot \mathbf{1}_{\{h_{n+1}\NOTE{\leq} h_{\min}\}},  \label{eq:AT}
\end{multline}
for $s\in[t_n,t_{n+1}]$, $n\in\mathbb{N}$,  $\widetilde{Y}(0)=X(0)$, and $\theta$ is as given in \eqref{eq:deftheta}. Thus the scheme is characterised by the sequence of tuples, $\big\{\big(\widetilde Y(s)\big)_{s\in[t_n,t_{n+1}]},h_{n+1}\big\}_{n\in\mathbb{N}}$.
The backstop map $\varphi:\mathbb{R}^d\times \mathbb{R}  \times\mathbb{R} \rightarrow \mathbb{R}^d$ in \eqref{eq:AT} satisfies for each $n \in \mathbb{N}$ 
\begin{multline}
\qquad \mathbb{E}\left[ \left\|X(s)- \varphi\left(\widetilde Y(t_n)\boldsymbol{,}\,\,t_n\boldsymbol{,}\,\, s-t_n\right)\right\|^2 \middle|\mathcal{F}_{t_n} \right] \leq  \left\|X(t_n) -\widetilde Y(t_n)  \right\|^2\\
 + C_{B_1} \int_{t_n}^{s} \mathbb{E}\left[\left\| X(r)-\varphi\left(\widetilde Y(t_n)\boldsymbol{,}\,\,t_n\boldsymbol{,}\,\, \NOTE{r-t_n}\right) \right\|^2 \middle|\mathcal{F}_{t_n}\right]dr+ C_{B_2} h_{\min}^3, \label{eq:defbackstop}
\end{multline}
 a.s, for positive constants $C_{B_1}$ and $C_{B_2}$.
\end{definition}


Throughout the article it is notationally convenient to make the following definition.
\begin{definition}
\label{def:Ytheta}
Let $\widetilde{Y}$ be as given in Definition \ref{def:adaptive explicit Milstein scheme} and define for each $n\in\mathbb{N}$
\begin{equation}\label{eq:Ytheta}
Y_{\theta}(s):=\theta\Big(\widetilde{Y}(t_n),t_n,s-t_n\Big),\quad s\in[t_n,t_{n+1}].
\end{equation}
\end{definition}

\begin{remark}\label{rem:com1}
The upper bound $\hmax$ prevents step sizes from becoming too large and allows us to examine strong convergence of the adaptive Milstein method \eqref{eq:AT} to solutions of \eqref{eq:true} as $\hmax \rightarrow 0$ (and hence as $\hmin \rightarrow 0$). Note that $\varphi$ satisfies \eqref{eq:defbackstop} if the backstop method satisfies a mean-square consistency requirement. In practice, instead of testing \eqref{eq:defbackstop}, we choose a backstop method that is strongly convergent with rate 1. 
\end{remark}
\begin{remark} \label{rem:com2}
For all $i=1,2,\dots, m$, $I_{i}^{t_n,t_{n+1}}$ in \eqref{eq:defISI} is a Wiener increment taken over a random step of length $h_{n+1}$ , which itself may depend on $\widetilde Y(t_n)$ and therefore is not necessarily independent and  normally distributed. However, since $h_{n+1}$ is  
$\mathcal{F}_{t_n}$-measurable by Assumption \ref{ass:h}, we have $I_{i}^{t_n,t_{n+1}}$ is $\mathcal{F}_{t_n}$-conditionally normally distributed and by the Optional Sampling Theorem (see for example~\cite{Shiryaev96}), for all $p=0,1,2,\dots$ 
\begin{align*}
\mathbb{E}\left[I_{i}^{t_n,t_{n+1}} \middle|\mathcal{F}_{t_n}\right]&=0,\quad a.s.;\numberthis \label{eq:Wmoments1n2}\\ \mathbb{E}\left[\left|I_{i}^{t_n,t_{n+1}}\right|^2 \middle|\mathcal{F}_{t_n}\right]&=h_{n+1},\quad a.s.;\numberthis\label{eq:WvarCond}\\
\mathbb{E}\left[\left|I_{i}^{t_n,s}\right|^{p} \middle|\mathcal{F}_{t_n}\right]&=\Cp{p}|s-t_n|^{\frac{p}{2}},\quad a.s.; \numberthis \label{eq:even Gaussian_p}
\end{align*}
where $\Cp{p}:=2^{p/2}\Gamma\left((p+1)/2 \right)\pi^{-1/2}$, and $\Gamma$ is the Gamma function (see for example \cite[p.148]{Papoulis}). In implementation, it is sufficient to replace the sequence of Wiener increments with i.i.d. $\mathcal{N} (0, 1)$ random
variables scaled at each step by the $\mathcal{F}_{t_n}$-measurable random variable $\sqrt{h_{n+1}}$. 
\end{remark}

We now provide a specific example of  a time-stepping strategy that we use in Section \ref{sec:num} and that satisfies the assumptions for our convergence proof in Theorem \ref{thm:result}. Suppose that for each $n=0,\dots, N-1$ and some fixed constant $\kappa>0$, we choose constant values $h_{\max}>h_{\min}>0$, $\rho>1$ such that $h_{\max}=\rho h_{\min}$ and
\begin{align}
 \NOTE{h_{n+1}=\hmin \vee \left(\hmax \wedge \frac{\hmax}{\big\|\widetilde Y(t_n)\big\|^{1/\kappa}} \right).}
\label{eq:defh}
\end{align}
Then \eqref{eq:hmaxhmin} in Assumption \ref{ass:h} holds for \eqref{eq:defh}. Notice also that, from \eqref{eq:defh}, the following bound applies on the event $\{\hmin < h_{n+1} \leq \hmax\}$:
\begin{align*}
    0\leq \big\|\widetilde Y(t_n)\big\| <  \left(\frac{\hmax}{\hmin}\right)^\kappa = \rho^\kappa.
\end{align*}
The strategy given by \eqref{eq:defh} is admissible in the sense given in \cite{kelly2018adaptive,kelly2021adaptive}. However, it also motivates the following class of time-stepping strategies to which our convergence analysis applies. 

\begin{definition}[Path-bounded time-stepping strategies]\label{def:defYn}
Let \Note{$\big\{\widetilde Y(t_n),h_{n+1} \big\}_{n\in\mathbb{N}}$ be a numerical approximation for \eqref{eq:true} given by \eqref{eq:AT}}, \Note{associated with a timestep sequence $\{h_{n+1}\}_{n\in\mathbb{N}}$ satisfying Assumption \ref{ass:h}}. We say that $\{h_{n+1}\}_{n\in\mathbb{N}}$ is a path-bounded time-stepping strategy for \eqref{eq:AT} if 
there exist real non-negative constants $0\leq Q<R$ (where $R$ may be infinite if $Q\neq 0$)
such that on the event $\{\hmin < h_{n+1} \leq \hmax\}$,
\begin{align}
Q\leq \big\|\widetilde Y(t_n)\big\| < R, \quad n=0,\dots, N-1. \label{eq:defYn}
\end{align}
\end{definition}
 Note that throughout this paper we use a strategy where $Q=0$ and $R<\infty$. As we will see in Section \ref{sec:telomere}, a careful choice of the parameter $\kappa$ can be used to minimise invocations of the backstop method when $\rho$ is fixed. 

\section{Main Results}
\label{subsec:converg}
Our first main result shows strong convergence 
with order 1 of solutions of \eqref{eq:AT} to solutions of \eqref{eq:true} when $\{h_{n+1}\}_{n\in\mathbb{N}}$ is a 
path-bounded time-stepping strategy ensuring that \eqref{eq:defYn} holds.
\begin{theorem}[Strong Convergence] \label{thm:result}
Let $(X(t))_{t\in[0,T]}$ be a solution of \eqref{eq:true} with initial value $X(0) = X_0\in\mathbb{R}^d$. Suppose that the conditions of Assumptions \ref{ass:f+g} and \ref{ass:SDEmoments_power} hold.\\ Let $\big\{\big(\widetilde Y(s)\big)_{s\in[t_n,t_{n+1}]},h_{n+1}\big\}_{n\in\mathbb{N}}$ be the adaptive Milstein scheme given in Definition \ref{def:adaptive explicit Milstein scheme} with initial value for the first component $\widetilde Y_0 = X_0$ and path-bounded time-stepping strategy $\{h_{n+1}\}_{n\in\mathbb{N}}$ satisfying the conditions of Definition \ref{def:defYn} for some $R<\infty$. Then there exists a constant  $C(R,\rho,T) > 0$ such that
\begin{align}\label{eq:mainEstimate}
\max_{t\in[0,T]}\Big(\mathbb{E}\Big[\|X(t)-\widetilde Y(t)\|^2\Big]\Big)^{1/2} \leq C(R,\rho,T)\,\hmax.
\end{align}
Furthermore,
\begin{equation}\label{eq:Clim}
\lim_{\rho\to\infty}C(R,\rho,T)=\infty.
\end{equation}
\end{theorem}
The proof of Theorem \ref{thm:result}, which is given in Section \ref{sec:proof_thm_4.1}, accounts for the properties of the random sequences $\{t_n\}_{n\in\mathbb{N}}$ and $\{h_{n+1}\}_{n\in\mathbb{N}}$ and uses \eqref{eq:defYn} to compensate for the non-Lipschitz drift and diffusion. 

Our second main result shows that for the specific strategy given by \eqref{eq:defh}, the probability of needing a backstop method can be made arbitrarily small by taking $\rho$ sufficiently large with a fixed $\kappa$.
\begin{theorem}[\textbf{Probability of Backstop}]\label{thrm:MarkovIneq}
Let all the conditions of Theorem \ref{thm:result} hold, and suppose that the path-bounded time-stepping strategy $\{h_{n+1}\}_{n\in\mathbb{N}}$ satisfies \eqref{eq:defh}.
Let $C(R,\rho,T)$ be the error constant in estimate \eqref{eq:mainEstimate} from the statement of Theorem \ref{thm:result}. 

For any fixed $\kappa\geq 1$ there exists a constant $C_{\text{prob}}=C_{\text{prob}}(T,R,\hmax)$ 
such that, for $h_{\max}\,\leq\,1/C(R,\rho,T)$,
\begin{equation}\label{eq:Prob_end}
    \prob{h_{n+1}=\hmin} \leq C_{\text{prob}}\,\,\rho^{1-2\kappa}.
\end{equation}
Further for arbitrarily small tolerance $\varepsilon\in(0,1)$, there exists $\rho>0$ such that
   $$\prob{h_{n+1}=\hmin} <\varepsilon,\quad n\in\mathbb{N}.$$
\end{theorem}
For proof see Section \ref{sec:proof_thm_4.2}.

\section{Numerical examples}\label{sec:num}
\begin{remark}\label{rem:lastStep}
We use the adaptive strategy in \eqref{eq:defh}. We ensure that we reach the final time by taking $h_{N}=T-t_{N-1}$ as our final step, and in a situation where this is smaller than $h_{\min}$
we use the backstop method (this is compatible with the proofs below).
\end{remark}

In the numerical experiments below, we set the \textit{adaptive Milstein scheme} (\texttt{AMil}) as in \eqref{eq:AT} with \eqref{eq:defh} as the choice of $h_{n+1}$. \textit{Projected Milstein} (\texttt{PMil})
in \cite[Eq. (24)]{beyn2017stochastic} is set to be the backstop method of \texttt{AMil} and the reference method of all models. Then we compare the strong convergence,looking at the root mean square (RMS) error, and efficiency, by comparing the CPU time, of \texttt{AMil} and \texttt{PMil}, \textit{Split-Step Backward Milstein} method (\texttt{SSBM}) \cite[Eq. (25)]{beyn2017stochastic}, the \textit{new variant of Milstein} (\texttt{TMil}) in \cite{KumarSabanis2019}, and the \textit{Tamed Stochastic Runge-Kutta of order $1.0$} (\texttt{TSRK1}) method \cite[Eq. (3.8) (3.9)]{gan2020tamed}. 
For the non-adaptive schemes, to examine strong convergence, we take as the fixed step $h_{\text{mean}}$
the average of all time steps over each path and each Monte Carlo realization $m =  1,\dots, M$ so that 
\Note{$$h_{\text{mean}}:=\frac{1}{M}\sum_{m=1}^{M}\frac{T}{N_{m}},
$$
where $N_{m}$ denotes the number of steps taken on the $m^{th}$ sample path to reach $T$.}

\subsection{One-dimensional test equations with multiplicative and additive noise}\label{sec:numerics1d}
In order to demonstrate strong convergence of order one for a scalar test equation with non-globally Lipschitz drift, consider
\begin{equation}\label{eq:1D_model}
dX(t)=(X(t)-3X(t)^3)dt+G(X(t))dW(t), \quad t\in[0,1].
\end{equation}
For illustrating both the multiplicative and additive noise cases, we estimate the RMS error by a Monte Carlo method using $M=1000$ trajectories for $h_{\max}=[2^{-14}, 2^{-12}, 2^{-10}, 2^{-8}, 2^{-6}]$, $\rho=2^2$, $\kappa=1$, and use as a reference solution \texttt{PMil} over a mesh with uniform step sizes $h_{\text{ref}}=2^{-18}$.

\begin{figure} 
    \centering

    \includegraphics[width=0.48\textwidth]{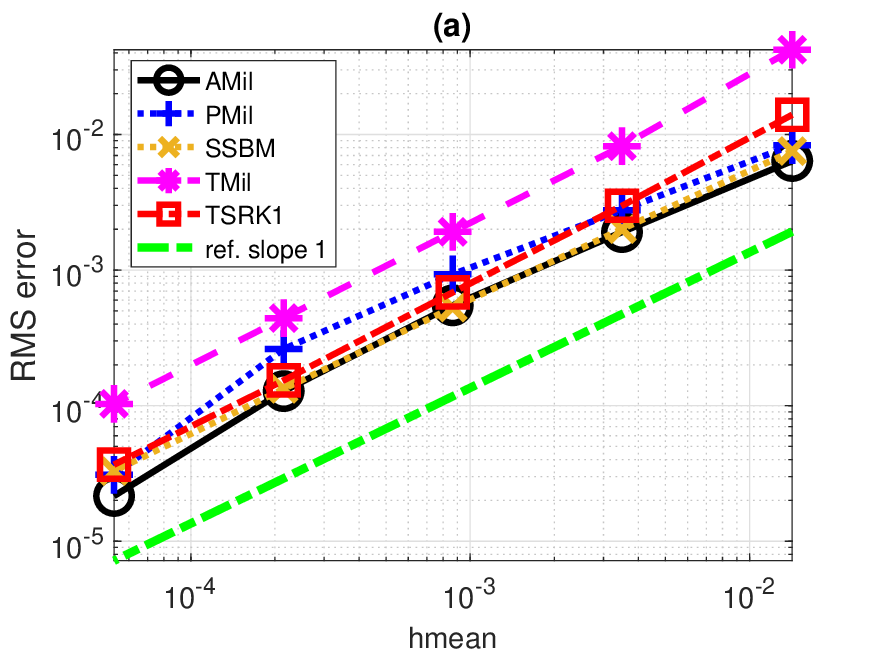}
    \includegraphics[width=0.48\textwidth]{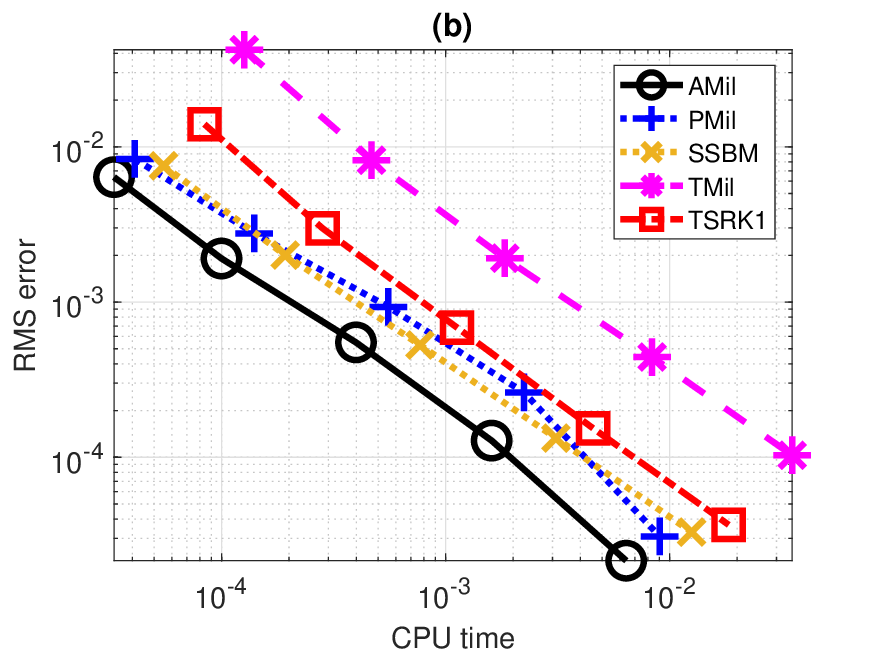}

    \includegraphics[width=0.48\textwidth]{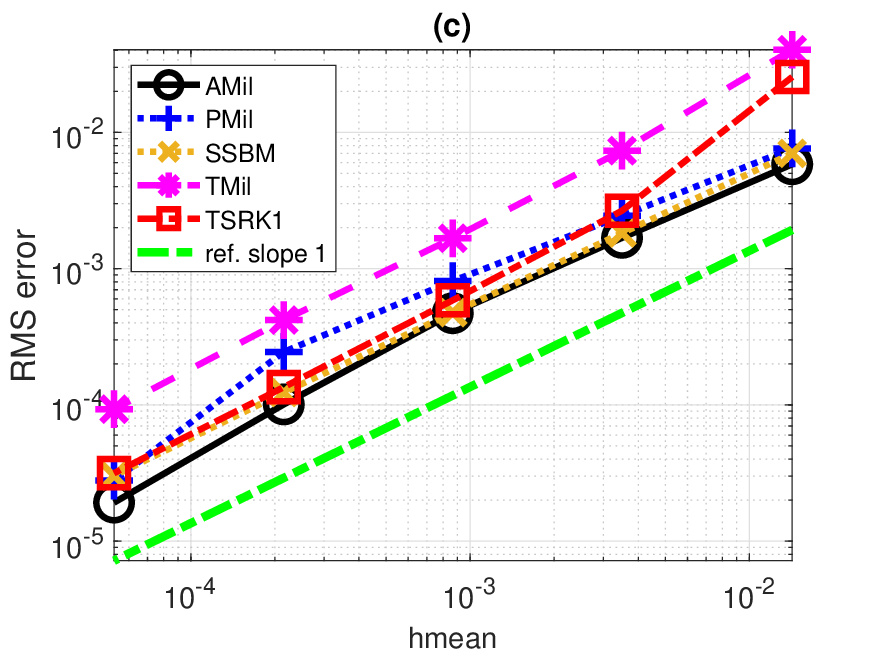}
    \includegraphics[width=0.48\textwidth]{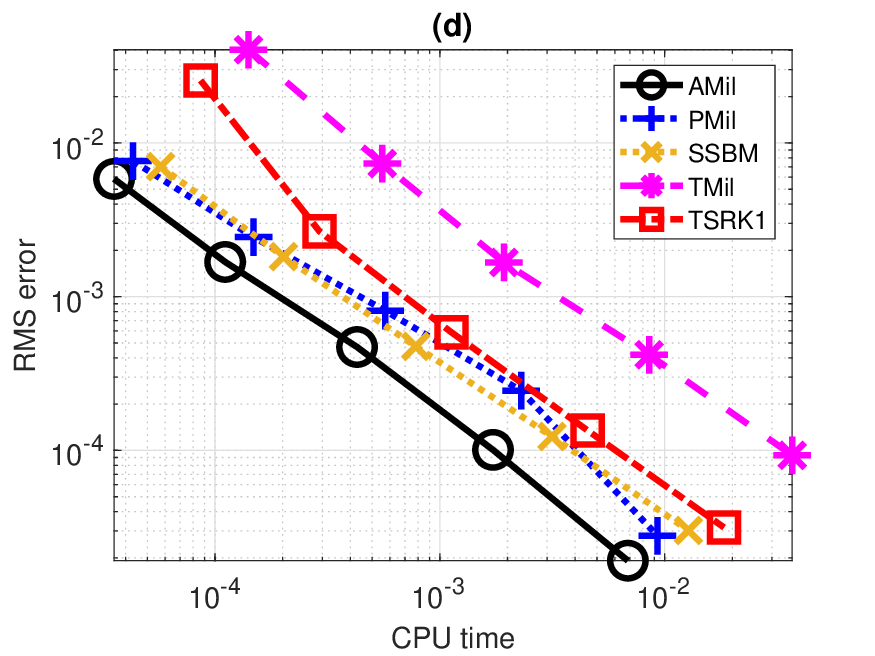}

    \includegraphics[width=0.48\textwidth]{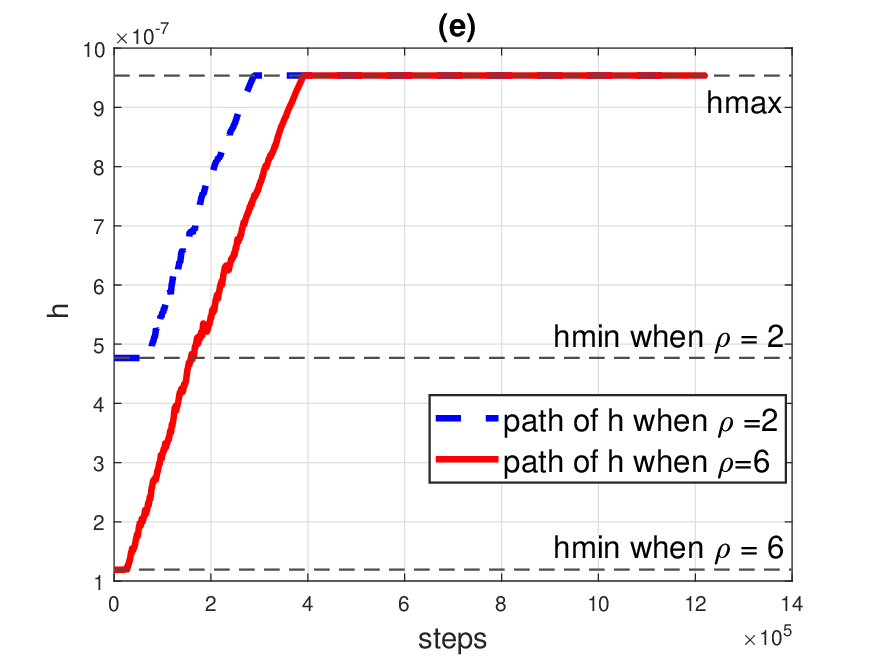}
    \includegraphics[width=0.48\textwidth]{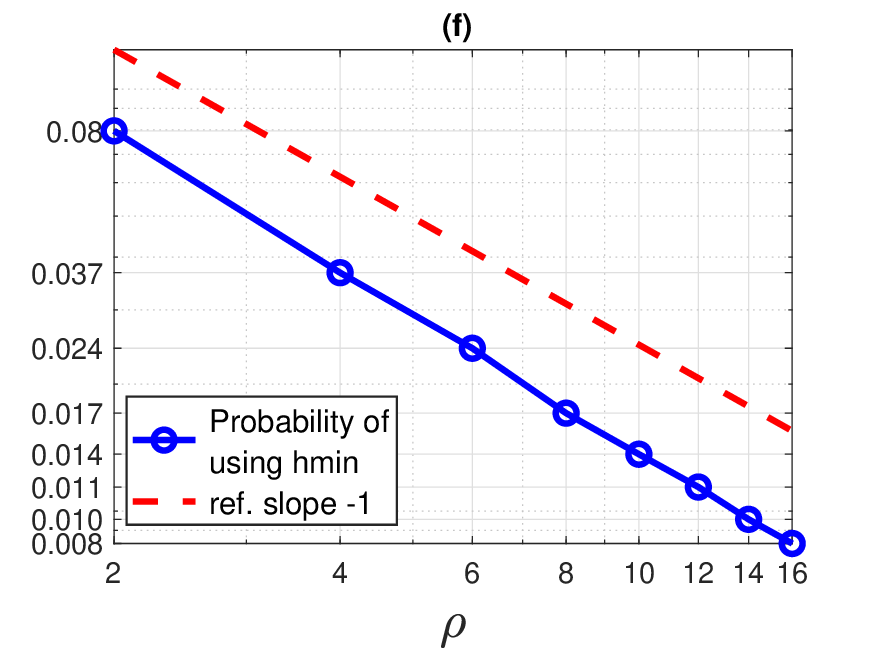}
    \caption{Strong convergence and efficiency of model \eqref{eq:1D_model} with (a) and (b) for additive noise; (c) and (d) for multiplicative noise. (e) Two paths of timestep $h$ for $\rho=2,6$ and in (f) 
    the estimated probability of using $\hmin$ for the multiplicative noise model with $M=100$ realizations.
    }
    \label{fig:scalarConv} 
\end{figure}

For additive noise we set $G(x)=\sigma$ in \eqref{eq:1D_model}, and for  multiplicative noise we set $G(x)=\sigma(1-x^2)$ with $\sigma=0.2$ and $X(0)=11$ in both cases. Strong convergence of order one is displayed by all methods in Figure \ref{fig:scalarConv} part (a) and (c) for the additive and multiplicative cases respectively, with the efficiency displayed in parts (b) and (d).

Finally, consider Theorem \ref{thrm:MarkovIneq}. We illustrate that the probability of our time-stepping strategy selecting $h_{\min}$, and therefore triggering an application of the backstop method, can be made arbitrarily small at every step by an appropriate choice of $\rho$ (with fixed $\kappa=1$).
Consider \eqref{eq:1D_model} again with $G(x)=\sigma(1-x^2)$, this time with
$X(0)=100$, $\kappa, T=1$,   $\hmax=2^{-20}$ and $\rho=[2, 4, \dots, 16]$. 
In Figure \ref{fig:scalarConv} (e), we plot two paths of $h$ when $\rho=2, 6$. Observe that when $\rho=2$ the backstop is triggered only for the first $10^5$ steps approximately, whereas once  $\rho$ is increase to $6$ this is reduced to the first $2\times 10^4$ steps approximately. Estimated probabilities of using $\hmin$ are plotted on a log-log scale as a function of $\rho$ in Figure \ref{fig:scalarConv} (f) (with $M=100$ realizations). The estimated probability of using $h_{\min}$ declines to zero as $\rho$ increases. We observe a rate close to $-1$, matching that in \eqref{eq:Prob_end} with $\kappa=1$.

\subsection{One-dimensional model of telomere shortening} \label{sec:telomere}
The following one-dimensional SDE model was given in \cite[Eq. (A6)]{grasman2011stochastic} for modelling the shortening over time of telomere length $L$ in DNA replication  
\begin{align}
    dL(t) = -\big(c+aL(t)^2\big)dt + \sqrt{\frac{1}{3}aL(t)^3}dW(t). \label{eq:TL}
\end{align}
The parameter $c$ determines the underlying decay rate of the length and $a$ controls the intensity at which random breaks occur in the telomere; we take $(a,c)=(0.41\times 10^{-6},7.5)$ as in \cite{grasman2011stochastic}. In this example we fix $\rho=4$, instead adjusting the parameter $\kappa$ in \eqref{eq:defh} to control use of the backstop method. Individual paths are shown in Figure \ref{fig:TL} where we take $\hmax=2^{-18}$, and $h=2^{-20}$ for the fixed step methods. 

We set $L(0)=1000$, noting from \cite{grasman2011stochastic} that initial values could be as high as (say) $L(0)=6000$ and remain physically realistic. The end of the interval of valid simulation is determined by the first time at which trajectories reach zero, and is therefore random. However this is not observed to occur in the timescale (25 days) we consider here.

By design \texttt{PMil} projects the data onto a ball of radius determined in part by the growth of the drift term. We see in Figure \ref{fig:TL} (a) that (\texttt{PMil}) immediately is reduced to approximately $200$.

Contrarily, the design of \texttt{TMil} scales both drift and diffusion terms by $1/(1+h|L|^2)$ for this model. When $h|L|^2$ is large this scaling can damp out changes from step to step, and in Figure \ref{fig:TL} (a) we see that \texttt{TMil} shows as (spuriously) almost constant. The paths of the other methods, \texttt{AMil}, \texttt{SSBM} and \texttt{TSRK1} are close together as shown in Figure \ref{fig:TL} (a) and in high detail in (b).   

Notice that we used $\kappa=8$ in \eqref{eq:defh} for \texttt{AMil} method to reduce the chance of requiring the backstop method \texttt{PMil} while keeping $\rho=4$. We avoid setting $\kappa=1$ in this case because $L(0)=1000$ and so the adaptive step $h_{n+1}$ would too frequently require the backstop method.

\begin{figure}
    \centering

    \includegraphics[width=0.48\textwidth]{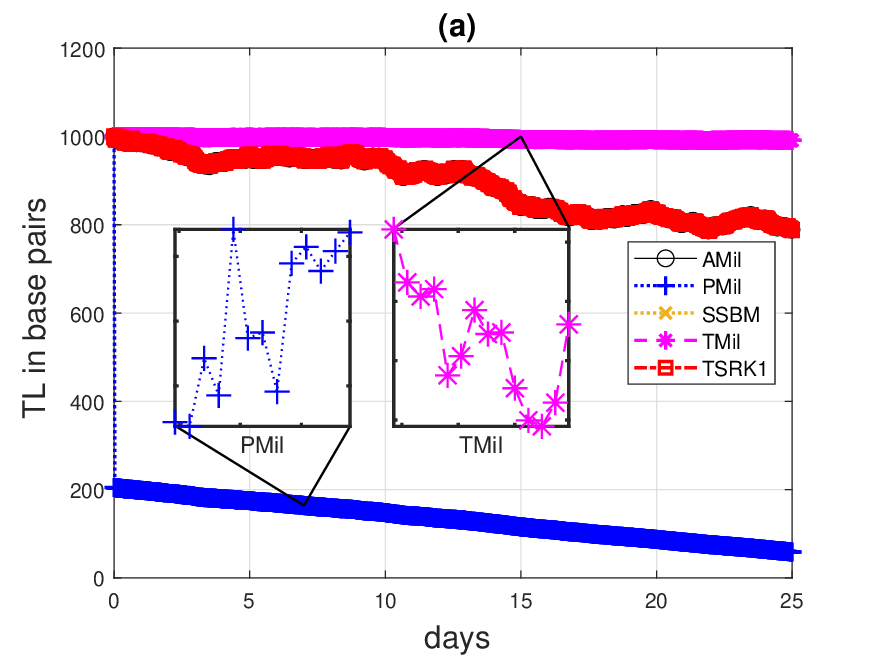}
    \includegraphics[width=0.48\textwidth]{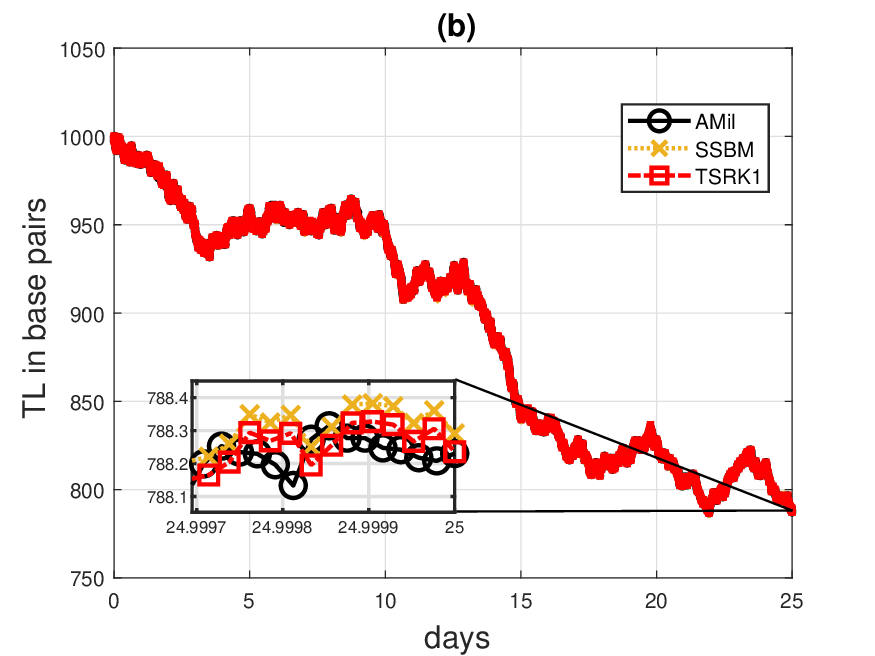}
    
    \caption{Single paths of the Telomere length SDE \eqref{eq:TL} solved over $25$ days. (b) shows a detailed plot from (a). 
    } 
    \label{fig:TL}
\end{figure}

\subsection{Two-dimensional test systems}
We now consider three ($i=1,2,3$) different SDEs: 
\begin{equation}\label{eq:2D_model}
dX(t)=F(X(t))dt+G_i(X(t))dW(t),\quad t\in[0,1],\quad X(0)=[7,9]^T,    
\end{equation}
with $W(t)=[W_1(t),W_2(t)]^T$, where $W_1$ and $W_2$ are independent scalar Wiener processes, $X(t)=[X_1(t),X_2(t)]^T$,  $F(x)=[x_2-3x_1^3,x_1-3x_2^3]^T$, and  
\begin{align*}
G_1(x)=\sigma\begin{pmatrix}x_1^2 &\quad 0\\ 0 &\quad x_2^2\end{pmatrix}, \, 
G_2(x)=\sigma\begin{pmatrix}x_2^2 &\quad x_2^2\\ x_1^2 &\quad x_1^2\end{pmatrix},\,
G_3(x)=\sigma\begin{pmatrix}1.5x_1^2 &\quad x_2\\ x_2^2 &\quad 1.5x_1\end{pmatrix}.
\end{align*}
$G_1$ is an example of diagonal noise, $G_2$  commutative noise, and $G_3$ non-commutative noise.

For $G_1$ and $G_2$ we use $\hmax=[2^{-14}, 2^{-12}, 2^{-10}, 2^{-8}, 2^{-6}]$, $h_{\text{ref}}=2^{-18}$, $\rho=4$ and $\kappa=1$. In Figure \ref{fig:2dConv} (a) and (c), we see order one strong convergence for all methods. Parts (b) and (d) show the efficiency of the adaptive method.

For $i=3$, the non-commutative noise case, take $\hmax=[2^{-8}, 2^{-7}, 2^{-6}, 2^{-5}, 2^{-4}]$, $ h_{\text{ref}}=2^{-11}$, $\rho=2^2$ and $X(0)=[3,4]^T$. To simulate the L\'evy areas we follow the method in \cite[Sec. 4.3]{higham2002maple},
which is based on the Euler approximation of a system of SDEs. Again, we observe order one convergence for all methods in Figure \ref{fig:2dConv} (e) and that 
\texttt{AMil} is the most efficient in (f). 
Note that as \texttt{TSRK1} 
is only supported theoretically for commutative noise we do not consider it here.

\begin{figure} 
    \centering
 
    \scalebox{0.91}{\includegraphics[width=0.48\textwidth]{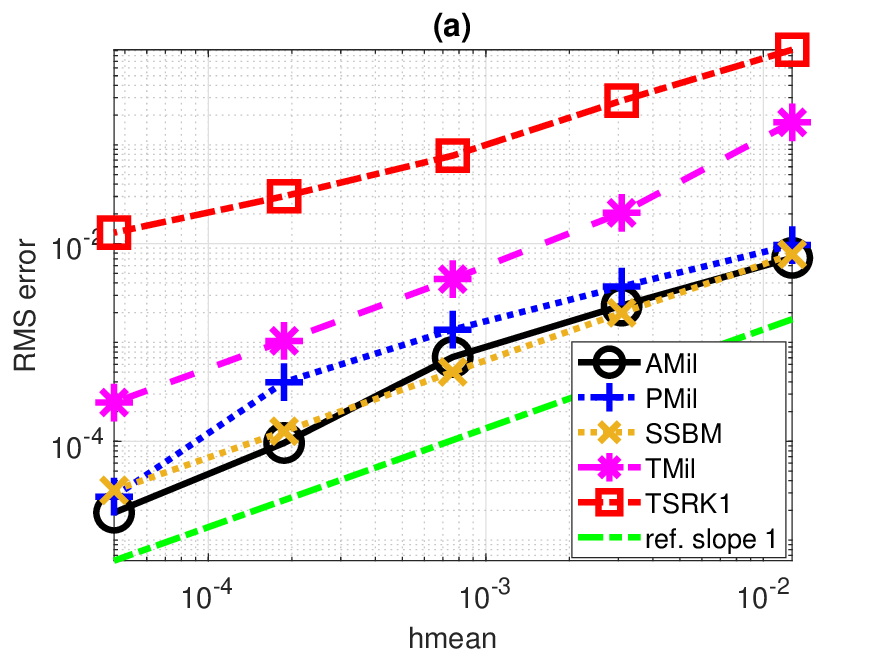}
    \includegraphics[width=0.48\textwidth]{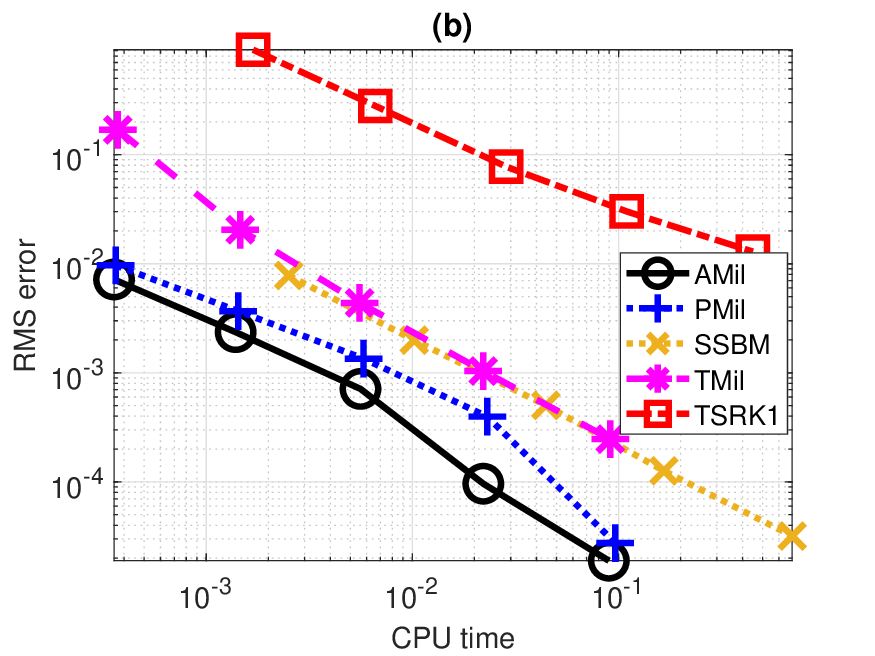}}

    \scalebox{0.91}{\includegraphics[width=0.48\textwidth]{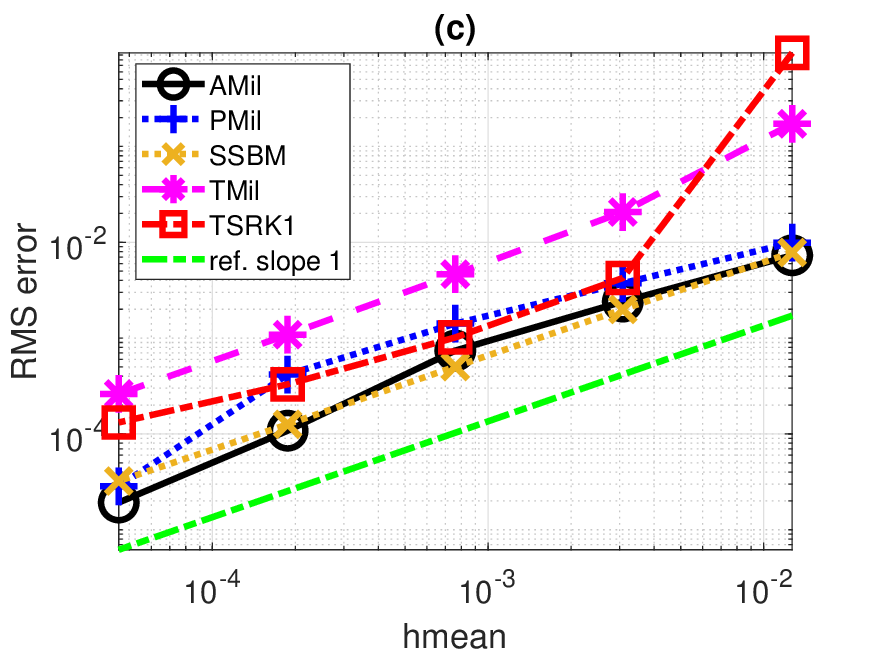}
    \includegraphics[width=0.48\textwidth]{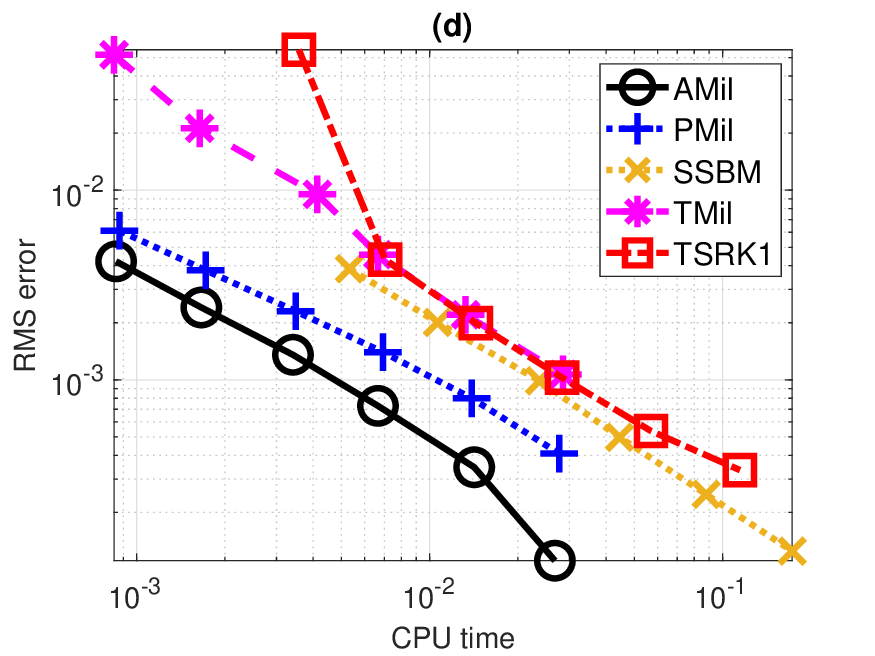}}

    \scalebox{0.91}{\includegraphics[width=0.48\textwidth]{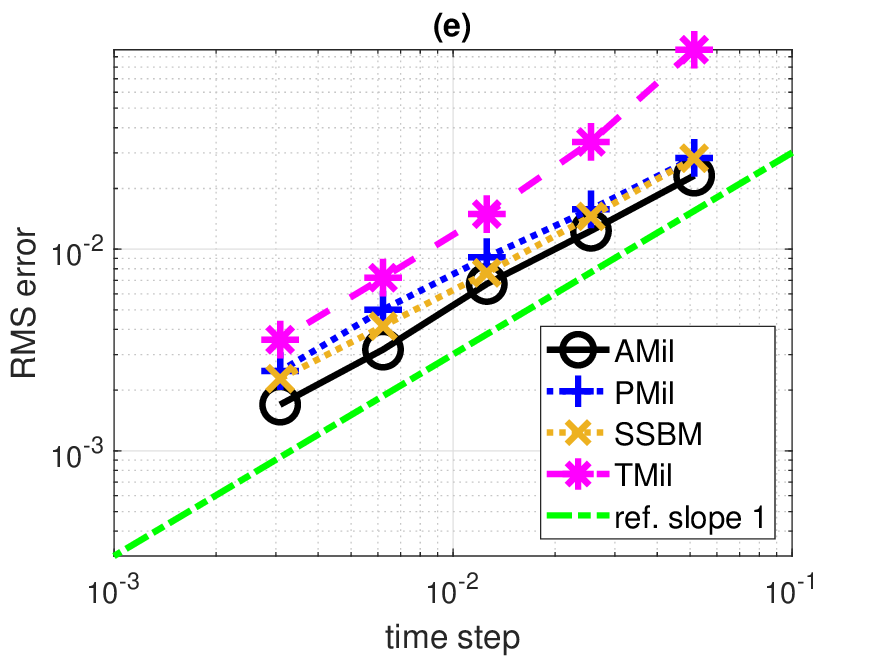}
    \includegraphics[width=0.48\textwidth]{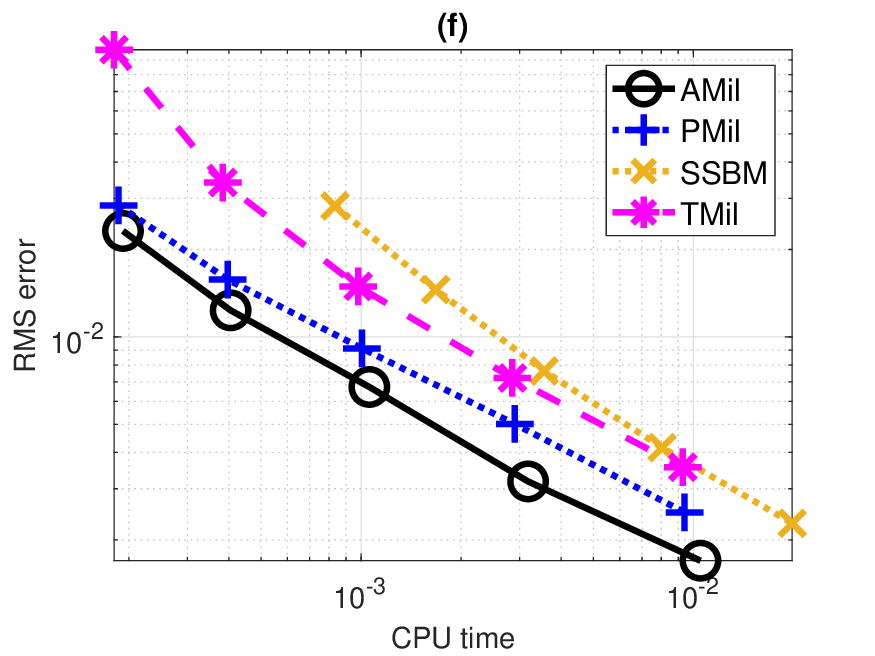}}
    
     \caption{Two-dimensional system \eqref{eq:2D_model}. (a) and (b) show the strong convergence and efficiency for diagonal noise, (c) and (d) with commutative and (e) and (f) for non-commutative noise. We choose $a=3$, $\sigma=0.2$ and  $b=1.5$.} 
    \label{fig:2dConv}
\end{figure}

\section{Preliminary Lemmas} \label{sec:lemmas}
We present five lemmas necessary for the proof of Theorem \ref{thm:result} and Theorem \ref{thrm:MarkovIneq}.
Throughout this section we assume that $f$ and $g$ satisfy Assumptions \ref{ass:f+g} and (except for Lemma \ref{lmm:PR}) that we are on the event $\{\hmin<h_{n+1}\leq\hmax\}$ so that \eqref{eq:defYn} holds \NoteTypo{of Definition \ref{def:defYn}}. We use \eqref{eq:Df+Dg}, \eqref{eq:D^2f+D^2g} and \eqref{eq:||f||+||g||} to define some bounded constant coefficients depending on $R<\infty$. The constant bounds in \eqref{eq:constants define} are then used in the development of a one-step error bound for the adaptive part of the scheme.
\begin{align}\label{eq:constants define}
\begin{split}
\big\|f\big(\widetilde Y(t_n)\big)\big\|  \leq \,& c_5(1+R^{q_1+2})=:C_{f};\\
 \big\|\mathbf{D}f\big(\widetilde Y(t_n)\big)\big\|_{\mathbf{F}}\leq\,& c_3(1+R^{q_1+1})=:C_{Df}; \\
\big\|g_i\big(\widetilde Y(t_n)\big)\big\|\leq\big\|g\big(\widetilde Y(t_n)\big)\big\|_{\mathbf{F}(d\times m)}\leq\,& c_6(1 + R^{q_2+2})=:C_{g_i}; \\
\big\|\mathbf{D}g_i\big(\widetilde Y(t_n)\big)\big\|_{\mathbf{F}} \leq\,& c_4(1 + R^{q_2+1})=:C_{Dg_i}. 
\end{split}
\end{align}

The following lemma provides a bound for the even conditional moments of the iterated stochastic integral in \eqref{eq:defISI}.
\begin{lemma} [\texttt{Iterated Stochastic Integral}]\label{lmm:ISI}
Let $\big\{\big(\widetilde Y(s)\big)_{s\in[t_n,t_{n+1}]},h_{n+1}\big\}_{n\in\mathbb{N}}$ be the adaptive Milstein scheme given in Definitions \ref{def:adaptive explicit Milstein scheme} and \ref{def:defYn}.
Then there exists a constant $C_{\texttt{ISI}}$ such that for $k\geq 1$, $n\in\mathbb{N}$ and $s\in[t_n,t_{n+1}]$, on the event $\{\hmin<h_{n+1}\leq\hmax\}$   
\begin{align}\label{eq:ISI}
    \mathbb{E}\Bigg[\Bigg\|\sum_{i,j=1}^{m}\mathbf{D}g_i\big(\widetilde Y(t_n)\big)g_j\big(\widetilde Y(t_n)\big)I_{j,i}^{t_n,s}\Bigg\|^{2k}\Bigg|\mathcal{F}_{t_n}\Bigg]\leq \Citerated{k,R}|s-t_n|^{2k},
\end{align}
where 
\begin{align*}
       \Citerated{k,R}:=&\,3^{2k} m^{4k}C_{Dg_i}^{2k} C_{g_i}^{2k}\Big( \Cp{4k}+1+\Cp{2k}^2+\Clevy{2k}\Big). \numberthis \label{eq:C_ISI}
\end{align*}
Here, $\Cp{p}$ is from \eqref{eq:even Gaussian_p}, $\Clevy{2k}$ is from Lemma \ref{lem:levy bound} with explicit form given in \eqref{eq:widehat_Ib}, and the $R$ dependence in $\Citerated{k,R}$ arises from \eqref{eq:constants define}.
\end{lemma}
\begin{proof}
First of all, for convenience we set
\begin{align*}
    G_{\texttt{ISI}}(s):=\Bigg\|\sum_{i,j=1}^{m}\mathbf{D}g_i\big(\widetilde Y(t_n)\big)g_j\big(\widetilde Y(t_n)\big)I_{i,j}^{t_n,s}\Bigg\|^{2k}.
\end{align*}
By \eqref{eq:Dgexpansion} and \eqref{eq:Jen_sum}, we have, for $s\in[t_n, t_{n+1}]$ and $n\in\mathbb{N}$,
\begin{multline*}
G_{\texttt{ISI}}(s)\leq 3^{2k-1}\Bigg(\Bigg\|\frac{1}{2}\sum_{i=1}^{m}\mathbf{D}g_i\big(\widetilde Y(t_n)\big)g_i\big(\widetilde Y(t_n)\big)\Big(\big(I_{i}^{t_n,s} \big)^2 -|s-t_n|\Big) \Bigg\|^{2k}\\
+\Bigg\|\frac{1}{2}\sum_{\substack{i,j=1\\i<j}}^{m}\Big(\mathbf{D}g_i\big(\widetilde Y(t_n)\big)g_j\big(\widetilde Y(t_n)\big)+\mathbf{D}g_j\big(\widetilde Y(t_n)\big)g_i\big(\widetilde Y(t_n)\big)\Big)I_{i}^{t_n,s}I_{j}^{t_n,s}\Bigg\|^{2k}\\
+\Bigg\|\sum_{\substack{i,j=1\\i<j}}^{m}\Big(\mathbf{D}g_i\big(\widetilde Y(t_n)\big)g_j\big(\widetilde Y(t_n)\big)-\mathbf{D}g_j\big(\widetilde Y(t_n)\big)g_i\big(\widetilde Y(t_n)\big)\Big)A_{ij}^{t_n,s} \Bigg\|^{2k}\Bigg).
\end{multline*}
Applying \eqref{eq:Jen_sum} again and by submultiplicativity of the Euclidean norm and the fact that the induced matrix 2-norm is bounded above by the Frobenius norm, for $s\in[t_n, t_{n+1}]$ and $n\in\mathbb{N}$, we get
\begin{multline*}
G_{\texttt{ISI}}(s) \leq 3^{2k-1}\Bigg(\frac{ m^{2k-1}}{2^{2k}}\sum_{i=1}^{m}\big\|\mathbf{D}g_i\big(\widetilde Y(t_n)\big)\big\|_{\mathbf{F}}^{2k}\big\|g_i\big(\widetilde Y(t_n)\big)\big\|^{2k}\Big(\big(I_{i}^{t_n,s} \big)^2+|s-t_n| \Big)^{2k}\\
+\left(\frac{ m(m-1)}{2}\right)^{2k-1}\sum_{\substack{i,j=1\\i<j}}^{m}\Big\|\mathbf{D}g_i\big(\widetilde Y(t_n)\big)g_j\big(\widetilde Y(t_n)\big)+\mathbf{D}g_j\big(\widetilde Y(t_n)\big)g_i(Y(t_n)\Big\|^{2k}\\
\times\bigg(\frac{1}{2^{2k}}\Big|I_{i}^{t_n,s}I_{j}^{t_n,s} \Big|^{2k}+\big|A_{ij}^{t_n,s} \big|^{2k}\bigg)\Bigg).
\end{multline*}
Applying conditional expectations on both sides, together with the pairwise conditional independence of $I_{i}^{t_n,s}$ and $I_{j}^{t_n,s}$ for $i\neq j$, \eqref{eq:Df+Dg} and \eqref{eq:constants define}, we have for $s\in[t_n, t_{n+1}]$ and $n\in\mathbb{N}$ 
\begin{multline*}
\mathbb{E}\Big[G_{\texttt{ISI}}(s)\Big|\mathcal{F}_{t_n}\Big] \leq 3^{2k}\Bigg( m^{2k-1} C_{Dg_i}^{2k} C_{g_i}^{2k}\sum_{i=1}^{m}\left(\mathbb{E}\left[\left|I_{i}^{t_n,s}\right|^{4k}\middle|\mathcal{F}_{t_n}\right]+|s-t_n|^{2k}  \right)\\
+\left(\frac{ m(m-1)}{2}\right)^{2k-1}C_{Dg_i}^{2k} C_{g_i}^{2k} \sum_{\substack{i,j=1\\i<j}}^{m} \bigg( \mathbb{E}\left[\left|I_{i}^{t_n,s}\right|^{2k}\middle|\mathcal{F}_{t_n}\right]\mathbb{E}\left[\left|I_{j}^{t_n,s}\right|^{2k}\middle|\mathcal{F}_{t_n}\right]\\
+\mathbb{E}\left[\left|A_{ij}^{t_n,s}\right|^{2k}\middle|\mathcal{F}_{t_n}\right]\bigg)\Bigg).
\end{multline*}
Using \eqref{eq:WvarCond}, \eqref{eq:even Gaussian_p} and \eqref{eq:levy moment} we have 
\begin{align*}
\mathbb{E}\Big[G_{\texttt{ISI}}(s)\Big|\mathcal{F}_{t_n}\Big]\leq& \Citerated{k,R} |s-t_n|^{2k}, 
\end{align*}
where $\Citerated{k,R}$ is in \eqref{eq:C_ISI}.
\end{proof}
The following lemma provides a bound on the conditional moments of the adaptive Milstein scheme in \eqref{eq:AT} over one step, in the case where the method applies the map $\theta$.
\begin{lemma}\label{lem:E[sup Ybar]}
\Note{Consider $\big\{\big(\widetilde Y(s)\big)_{s\in[t_n,t_{n+1}]},h_{n+1}\big\}_{n\in\mathbb{N}}$ from  Definitions \ref{def:adaptive explicit Milstein scheme} and \ref{def:defYn}, and let\\ $(Y_{\theta}(s))_{s\in(t_n,t_{n+1}]}$} be as defined in Definition \eqref{def:Ytheta}. Then there exists a constant $C_{ Y_{\theta}}> 0$ such that for $k\geq 1$, $n\in\mathbb{N}$ and $s\in\Note{(}t_n, t_{n+1}]$, on the event $\{\hmin<h_{n+1}\leq\hmax\}$, 
\begin{align}
 \mathbb{E}\Big[\big\| Y_{\theta}(s)\big\|^k\Big|\mathcal{F}_{t_n}\Big] \leq  C_{ Y_{\theta}}\big(k,R\big), \label{eq:sup Ybar}
\end{align}
where
\begin{align}
    C_{ Y_{\theta}}\big(k,R\big):=\,\,4^{k-1}\Big(R^k+C_f^k+m^k C_{g_i}^k\,\,  \Cp{k}+\Citerated{2k}^{1/2}\Big), \numberthis \label{eq:C_barY}
\end{align}
with the constant $C_{\texttt{ISI}}$ from Lemma \ref{lmm:ISI}.
\end{lemma}
\begin{proof}
By \eqref{eq:Ytheta}, \eqref{eq:deftheta}
and \eqref{eq:Jen_sum}, we have, for $s\in\Note{(}t_n, t_{n+1}]$ and $n\in\mathbb{N}$,
\begin{align*}
    \big\| Y_{\theta}(s)\big\|^k=\,&\left\|\theta\left(\widetilde Y(t_n)\boldsymbol{,}\,\, t_n\boldsymbol{,}\,\, s-t_n\right) \right\|^k \\
    \leq &\,\,4^{k-1}\Bigg(\big\|\widetilde Y(t_n)\big\|^k+\left\|f\big(\widetilde Y(t_n)\big)\right\|^k|s-t_n|^k+\left\|\sum_{i=1}^{m}g_i\big(\widetilde Y(t_n)\big)I_{i}^{t_n,s}\right\|^{k}\\
    &\quad+\Bigg(\Bigg\|\sum_{i,j=1}^{m}\mathbf{D}g_i\big(\widetilde Y(t_n)\big)g_j\big(\widetilde Y(t_n)\big)I_{j,i}^{t_n,s}\Bigg\|^{2k}\Bigg)^{1/2}\Bigg).
\end{align*}
Applying \eqref{eq:Jen_sum}, \eqref{eq:defYn} and \eqref{eq:constants define} for $s\in\Note{(}t_n, t_{n+1}]$ and $n\in\mathbb{N}$, it yields
\begin{align*}
    \big\| Y_{\theta}(s)\big\|^k
    \leq &\,\,4^{k-1}\bigg(R^k+C_f^k |s-t_n|^k+m^{k-1} C_{g_i}^k \sum_{i=1}^{m}\left|I_{i}^{t_n,s}\right|^{k}\\
    &\quad+\Bigg(\Bigg\|\sum_{i,j=1}^{m}\mathbf{D}g_i\big(\widetilde Y(t_n)\big)g_j\big(\widetilde Y(t_n)\big)I_{j,i}^{t_n,s}\Bigg\|^{2k}\Bigg)^{1/2}\Bigg).
\end{align*}
Taking conditional expectation on both sides, with Jensen's inequality on the last term we have for $s\in\Note{(}t_n, t_{n+1}]$ and  $n\in\mathbb{N}$  
\begin{multline*}
\mathbb{E}\Big[\|Y_{\theta}(s)\|^k\Big|\mathcal{F}_{t_n}\Big]\leq\, 4^{k-1}\bigg(R^k+C_f^k|s-t_n|^k+  m^{k-1} C_{g_i}^k\sum_{i=1}^{m}\mathbb{E}\left[\left|I_{i}^{t_n,s}\right|^{k}\middle|\mathcal{F}_{t_n} \right] \\
+\left(\mathbb{E}\Bigg[\Bigg\|\sum_{i,j=1}^{m}\mathbf{D}g_i\big(\widetilde Y(t_n)\big)g_j\big(\widetilde Y(t_n)\big)I_{j,i}^{t_n,s}\Bigg\|^{2k}\Bigg|\mathcal{F}_{t_n}\Bigg]\right)^{1/2}\Bigg).
\end{multline*}
Using \eqref{eq:WvarCond}, \eqref{eq:ISI} from Lemma \ref{lmm:ISI} and since $|s-t_n|\leq h_{\max}\leq 1$ \eqref{eq:hmaxhmin} we have 
\begin{align*}
     \mathbb{E}\Big[\|Y_{\theta}(s)\|^k\Big|\mathcal{F}_{t_n}\Big] \leq C_{Y_{\theta}}(k,R),
\end{align*}
where $C_{Y_{\theta}}(k,R)$ is in \eqref{eq:C_barY}.
\end{proof}
The following lemma
proves regularity in time of the adaptive Milstein scheme in \eqref{eq:AT} when applying the map $\theta$.
\begin{lemma}[\texttt{Scheme Regularity}]\label{lmm:||eror||^2n}
\Note{Consider $\big\{\big(\widetilde Y(s)\big)_{s\in[t_n,t_{n+1}]},h_{n+1}\big\}_{n\in\mathbb{N}}$ in  Definitions \ref{def:adaptive explicit Milstein scheme} and \ref{def:defYn}, and let $(Y_{\theta}(s))_{s\in\Note{(}t_n,t_{n+1}]}$} be as defined in Definition \eqref{def:Ytheta}.
Then there exists a constant $C_{\texttt{SR}}$ such that  for $k\geq 1$, $n\in\mathbb{N}$ and $s\in\Note{(}t_n, t_{n+1}]$, on the event $\{\hmin<h_{n+1}\leq\hmax\}$
\begin{align}
\mathbb{E}\Big[\big\| Y_{\theta}(s)- \widetilde Y(t_n) \big\|^{2k} \Big|\mathcal{F}_{t_n}\Big] \leq \Cscheme{k,R} |s-t_n|^k, \label{eq:||eror||^2n} 
\end{align}
where
\begin{align*}
      \Cscheme{k,R}:= \,\, 3^{2k-1}\Big(C_f^{2k} +m^{2k} C_{g_i}^{2k}\,\,  \Cp{2k} +\Citerated{2k} \Big),\numberthis \label{eq:C_Y_n}
\end{align*}
with the constant $C_{\texttt{ISI}}$ from Lemma \ref{lmm:ISI}.
\end{lemma}
\begin{proof}
The method of proof is
similar to the proof of Lemma \ref{lem:E[sup Ybar]}.
\end{proof}

\begin{remark}\label{rem:supScriNot}
\Note{Our analysis requires a certain number of finite moments for the SDE \eqref{eq:true}, and it is necessary to track exactly what those are in order to see that the conditions of Assumption \ref{ass:SDEmoments_power} are not violated. To this end, we introduce a superscript notation for random variables appearing as conditional expectations at this point. The notation should be interpreted according to the following example: in \eqref{eq:C_PR_bar} below the random variable  $C_{\texttt{PR}}^{\{2k(q+2)\}}$ requires $2k(q+2)$ finite moments of the SDE \eqref{eq:true} to have finite expectation.}
\end{remark}

The following lemma examines the regularity of solutions of the SDE \eqref{eq:true}. 
\begin{lemma} [\texttt{Path Regularity}] \label{lmm:PR}
Let $f$, $g$ also satisfy Assumption \ref{ass:SDEmoments_power}, and let $(X(s))_{s\in[t_n,t_{n+1}]}$ be a solution of \eqref{eq:true}. Then there exists an $\mathcal{F}_{t_n}$-measurable random variable $\overline C_{\texttt{PR}}^{\{2k(q+2)\}}$ such that for $k\geq 1$, $n\in\mathbb{N}$ and $s\in[t_n,t_{n+1}]$  a.s.
\begin{align}
     \mathbb{E}\Big[\|X(s)-X(t_n)\|^{2k}\Big|\mathcal{F}_{t_n}\Big] &\,\leq\,\, \CpathBar{2k(q+2)}{k}\,\,|s-t_n|^{k}, \label{eq:PathRegularity_bar}
\end{align}
where $q=q_1\vee  q_2$ is as defined in Assumption \ref{ass:SDEmoments_power}. Where a.s.
\begin{multline*}
    \CpathBar{2k(q+2)}
    :=2^{4k-2}c_5^{2k}\left(1+\mathbb{E}\left[\Note{\sup_{p\in[t_n,t_{n+1}]}}\|X(p)\|^{2k(q_1+2)}\middle|\mathcal{F}_{t_n}\right]\right)\\
    +2^{4k-2}(k(2k-1))^k c_6^{2k}\left(1+\mathbb{E}\left[\Note{\sup_{p\in[t_n,t_{n+1}]}}\|X(p)\|^{2k(q_2+2)}\middle|\mathcal{F}_{t_n}\right]\right) \numberthis\label{eq:C_PR_bar}  
\end{multline*}
 where the expectation of $\overline C_\texttt{PR}^{\{2k(q+2)\}}$ is denoted $\Cpath{k}$, given by
\begin{equation}
\Cpath{k}:=\mathbb{E}\left[\CpathBar{2k(q+2)}{k}\right]\leq 2^{4k-2}\left(1+\Cx \right)\big( c_5^{2k}+(k(2k-1))^k c_6^{2k}\big). \label{eq:E[C_PR_bar]}  
\end{equation}
\end{lemma}
\begin{proof}
The method of proof follows that of \cite[Thm. 7.1]{mao2007SDEapp}. The bound \eqref{eq:E[C_PR_bar]} follows from \eqref{eq:SDEmoments} and Assumption \ref{ass:SDEmoments_power}.
\end{proof}
The following lemma provides a bound on the even conditional moments of the remainder term from a Taylor expansion of either the drift $f$ or diffusion $g$, around $\widetilde{Y}(t_n)$.
\begin{lemma}[\texttt{Taylor Error}] \label{lmm:T01} 
\Note{Consider $\big\{\big(\widetilde Y(s)\big)_{s\in[t_n,t_{n+1}]},h_{n+1}\big\}_{n\in\mathbb{N}}$ from\\ Definitions \ref{def:adaptive explicit Milstein scheme} and \ref{def:defYn}, and let $(Y_{\theta}(s))_{s\in[t_n,t_{n+1}]}$ be as defined in Definition \eqref{def:Ytheta}.}
Let $u\in \{f, g\}$ 
and set $c_{\mathbf{D}2}:=c_1\vee c_2$. Then there exists a constant $C_{\texttt{TE}}$ such that  for $k\geq 1$, $n\in\mathbb{N}$ and $s\in[t_n, t_{n+1}]$, on the event $\{\hmin<h_{n+1}\leq\hmax\}$,  
\begin{align*}
    \mathbb{E}\left[\Big\|\int_{0}^{1}(1-\epsilon)\mathbf{D}^2u\Big(\widetilde Y(t_n) -\epsilon\big(  Y_{\theta}(s)-\widetilde Y(t_n)\big)\Big)   d\epsilon \Big\|_{\mathbf{T_3}}^{2k}\middle| \mathcal{F}_{t_n}\right]\leq C_{\texttt{TE}}\big(k,R\big ), \numberthis \label{eq:T01}
\end{align*}
where $C_{\texttt{TE}}\left(k,R \right):=c_{\mathbf{D}2}^{2k}\left(1+ 3^{2kq+1}\left(R^{2kq}+ C_{Y_{\theta}}\left(k,R \right)\right) \right)$,
where $C_{Y_{\theta}}\big(k,R\big)$ is from Lemma \ref{lem:E[sup Ybar]}.
\end{lemma}
\begin{proof}
By using \eqref{eq:Jen_int}, \eqref{eq:Jen_sum}, \eqref{eq:D^2f+D^2g}, Lemma \ref{lem:E[sup Ybar]}, \eqref{eq:defYn} and since $c_{\mathbf{D}2}=c_1 \vee c_2$, $q=q_1 \vee q_2$ we have
\begin{align*}
&\mathbb{E}\left[\Big\|\int_{0}^{1}(1-\epsilon)\mathbf{D}^2u\Big(Y(t_n) -\epsilon\big(Y_{\theta}(s)-\widetilde Y(t_n)\big)\Big)   d\epsilon \Big\|_{\mathbf{T_3}}^{2k}\middle| \mathcal{F}_{t_n}\right]\\
\leq&\mathbb{E}\left[\int_{0}^{1}(1-\epsilon)^{2k}\Big\|\mathbf{D}^2u\Big(\widetilde Y(t_n) -\epsilon\big(Y_{\theta}(s)-\widetilde Y(t_n)\big)\Big)\Big\|_{\mathbf{T_3}}^{2k}   d\epsilon \middle| \mathcal{F}_{t_n}\right]\\
\leq& c_{\mathbf{D}2}^{2k}\mathbb{E}\left[\int_{0}^{1}(1-\epsilon)^{2k}\Big(1+\big\|\widetilde Y(t_n)-\epsilon\cdot\big(Y_{\theta}(s)-\widetilde Y(t_n)\big)\big\|^{2kq}\Big) d\epsilon\middle |\mathcal{F}_{t_n}\right] \\
\leq&  c_{\mathbf{D}2}^{2k}\mathbb{E}\Big[1+3^{2kq-1}\big\|\widetilde Y(t_n)\big\|^{2kq}\\
&\qquad\qquad+\int_{0}^{1}(1-\epsilon)^{2k}\epsilon^{2kq}3^{2kq}\left(\|Y_{\theta}(s)\|^{2kq}+\big\|\widetilde Y(t_n)\big\|^{2kq}\right) d\epsilon\Bigg |\mathcal{F}_{t_n}\Bigg]\\
\leq& c_{\mathbf{D}2}^{2k}\left(1+ 3^{2kq+1}\big\|\widetilde Y(t_n)\big\|^{2kq}+3^{2kq}\mathbb{E}\Big[\|Y_{\theta}(s)\|^{2kq} \Big|\mathcal{F}_{t_n}\Big]\right) \\
=& \Ctaylor{k,R},
\end{align*}
where $(1-\epsilon)^{2k}\,\epsilon^{2kq}\leq 1$ for $k,q\geq 1$ and $\epsilon\in[0,1]$.
\end{proof}

\section{Proof of Main Theorems}
\label{sec:proof}
In this section we prove the strong convergence result of Theorem \ref{thm:result} and Theorem \ref{thrm:MarkovIneq} on the probability of using the backstop and the role of $\rho$. 
\subsection{Setting up the error function}
Notice that $\widetilde Y (s)$, from the explicit adaptive Milstein scheme \eqref{eq:AT}, takes either the Milstein map $\theta$ in \eqref{eq:deftheta} or the backstop map $\varphi$ in \eqref{eq:defbackstop} depending on the value of $h_{n+1}$.
Thus, we define the error by
\begin{align}
   \widetilde E(s):=\,X(s)-\widetilde Y (s) = E_{\theta}(s) + E_{\varphi}(s), \label{eq:defE(r)_combined} 
\end{align}
for $s\in[t_n, t_{n+1}]$ and $n\in\mathbb{N}$. Here 
\begin{align}
    E_{\varphi}(s):=\left(X(s)- \varphi\left(\widetilde Y(t_n)\boldsymbol{,}\,\,t_n\boldsymbol{,}\,\, \Note{s-t_n}\right)\right)\, \mathbf{1}_{\{h_{n+1}\leq h_{\min}\}}, \label{eq:defE_varphi}
\end{align}
and $Y_{\theta}(s)$ is as defined in Definition  \ref{def:Ytheta} and
\begin{eqnarray}
E_{\theta}(s)&:=& \big(X(s) -  Y_{\theta}(s)\big)\, \mathbf{1}_{\{h_{\min}<h_{n+1}\leq h_{\max}\}}\nonumber\\
&=&\,\left(\widetilde E(t_n)+\int_{t_n}^{s}\Delta f\big(X(r),\widetilde Y(t_n)\big)dr\right.\nonumber\\
&&\left.+\sum_{i=1}^{m}\int_{t_n}^{s}\Delta g_i\big(r,X(r),\widetilde Y(t_n)\big)dW_i(r)\right)\, \mathbf{1}_{\{h_{\min}<h_{n+1}\leq h_{\max}\}}, 
\label{eq:defE(r)}
\end{eqnarray}
with
\begin{align*}
\Delta f\big(X(r),\widetilde  Y(t_n)\big):=\,&f(X(r))- f\big(\widetilde Y(t_n)\big); \numberthis\label{eq:errorF}\\
\Delta g_i\big(r,X(r),\widetilde Y(t_n)\big):=\,&g_i(X(r))-g_i\big(\widetilde Y(t_n)\big)-\sum_{j=1}^{m}\mathbf{D}g_i\big(\widetilde Y(t_n)\big)g_j\big(\widetilde Y(t_n)\big)I_{j}^{t_n,r}. \numberthis \label{eq:errorG}
\end{align*}
To simplify the proofs of Theorems \ref{thm:result} and \ref{thrm:MarkovIneq}, we require two Lemma  below. 
First, we find the second-moment bound of $\Delta g_i$ in \eqref{eq:errorG} on the event $\{\hmin<h_{n+1}\leq\hmax\}$ (so that \eqref{eq:defYn} holds).
\begin{lemma} \label{lmm:Gr}
Let $g$ satisfy Assumption \ref{ass:f+g} and $\Delta g_i$ be as in \eqref{eq:errorG}. Take $s\in[t_n,t_{n+1}]$, let 
$X(s)$
be a solution of \eqref{eq:true}, \Note{consider $\big(\widetilde Y(s),h_{n+1}\big)$ from  Definitions \ref{def:adaptive explicit Milstein scheme} and \ref{def:defYn}, and let $Y_{\theta}(s)$} be as defined in Definition \ref{def:Ytheta}. 
In this case there exists a constant $C_{G}$ such that, on the event $\{h_{\min}<h_{n+1}\leq h_{\max}\}$, 
\begin{multline*}
    \qquad\mathbb{E}\left[\left\|\Delta g_i\big(s,X(s),\widetilde Y(t_n)\big) \right\|^2\middle|\mathcal{F}_{t_n}\right]\\
    \leq  2\mathbb{E}\Big[\big\|g(X(s))-g\big(Y_{\theta}(s)\big)\big\|_{\mathbf{F}(d\times m)}^2 \Big|\mathcal{F}_{t_n}\Big] +C_{G}(R)  |s-t_n|^2, \numberthis \label{eq:lmmGr}
\end{multline*}
where 
\begin{align*}
    C_{G}(R) :=\,8 C_{Dg_i}^2 \big(C_f^2+ \Citerated{1,R}\big) +4\Ctaylor{2,R}^{1/2} \Cscheme{4,R}^{1/2}, \numberthis \label{eq:C_Gr} 
\end{align*}
and $C_{\texttt{ISI}}$, $C_{\texttt{TE}}$ and $C_{\texttt{SR}}$ are from Lemma \ref{lmm:ISI}, \ref{lmm:T01} and \ref{lmm:||eror||^2n}, respectively.
\end{lemma}
\begin{proof}
We first substitute $\Delta g_i$ by \eqref{eq:errorG} in the LHS of \eqref{eq:lmmGr}, then add in and subtract out $g_i\big( Y_{\theta}(s)\big)$,
by \eqref{eq:Jen_sum} we have
\begin{multline}
    \qquad\mathbb{E}\left[\left\|\Delta g_i\big(s,X(s),\widetilde Y(t_n)\big) \right\|^2\middle|\mathcal{F}_{t_n}\right]
    \leq  2\mathbb{E}\Big[\Big\|g_i(X(s))-g_i\big( Y_{\theta}(s)\big)\Big\|^2 \Big|\mathcal{F}_{t_n}\Big] \\
     +\underbrace{2\mathbb{E}\Bigg[\Bigg\|g_i\big( Y_{\theta}(s)\big)-g_i\big(\widetilde Y(t_n)\big)-\sum_{j=1}^{m}\mathbf{D}g_i\big(\widetilde Y(t_n)\big)g_j\big(\widetilde Y(t_n)\big)I_{j}^{t_n,s} \Bigg\|^2 \Bigg|\mathcal{F}_{t_n}\Bigg]}_{=: G_{1}}.\label{eq:G_1}
\end{multline}
To analyse $G_{1}$, we expand $g_i ( Y_{\theta}(s) )$ using Taylor's theorem (see for example \cite[A.1]{lord2014introduction}) around $g_i\big(\widetilde Y(t_n)\big)$ to get
\begin{multline}
    \qquad g_i\big( Y_{\theta}(s)\big)-g_i\big(\widetilde Y(t_n)\big) =\, \mathbf{D}g_i\big(\widetilde Y(t_n)\big)\big(  Y_{\theta}(s)-\widetilde Y(t_n)\big)\\
    +\int_{0}^{1}(1-\epsilon)\mathbf{D}^2g_i\Big(\widetilde Y(t_n)-\epsilon\big(  Y_{\theta}(s) -\widetilde Y(t_n)\big)\Big)\Big[  Y_{\theta}(s)-\widetilde Y(t_n)\Big]^2d\epsilon, \label{eq:TaylorG}
\end{multline}
where we recall from Section \ref{sec:prelim} that $[\cdot]^2$ represents the outer product of a vector with itself.
Substituting \eqref{eq:TaylorG} into $G_{1}$ in \eqref{eq:G_1}, then taking out $\mathbf{D}g_i\big(\widetilde Y(t_n)\big)$ as a common factor, and applying \eqref{eq:Jen_sum} gives 
\begin{align*}
    G_{1}\,\leq & \,\,4 \mathbb{E}\Bigg[\Bigg\|\mathbf{D}g_i\big(\widetilde Y(t_n)\big)\bigg(  Y_{\theta}(s)-\widetilde Y(t_n)-\sum_{j=1}^{m}g_j\big(\widetilde Y(t_n)\big)I_{j}^{t_n,s}\bigg)\Bigg\|^2\Bigg|\mathcal{F}_{t_n}\Bigg] \\
    &\quad+4  \mathbb{E}\Bigg[\Bigg\|\int_{0}^{1}(1-\epsilon)\mathbf{D}^2g_i\Big(\widetilde Y(t_n)-\epsilon\big(  Y_{\theta}(s)-\widetilde Y(t_n)\big)\Big) \Big[  Y_{\theta}(s)-\widetilde Y(t_n)\Big]^2d\epsilon \Bigg\|^2 \Bigg|\mathcal{F}_{t_n}\Bigg] \\
    =:&\,\, G_{1.1}+G_{1.2}. \numberthis\label{eq:G_1_aim}
\end{align*}
For $G_{1.1}$ in \eqref{eq:G_1_aim}, by submultiplicativity of the Euclidean norm and the fact that the induced matrix 2-norm is bounded above by the Frobenius norm; by \eqref{eq:Ytheta}, \eqref{eq:constants define} and \eqref{eq:ISI} in the statement of Lemma \ref{lmm:ISI} with $k=1$, we have
\begin{align*}
    G_{1.1}
     \,\leq\,\,&  8\Big\|\mathbf{D}g_i\big(\widetilde Y(t_n)\big)\Big\|_{\mathbf{F}}^2\Bigg(\Big\|f\big(\widetilde Y(t_n)\big)\Big\|^2|s-t_n|^2\\
     &\qquad\qquad+\mathbb{E}\Bigg[\Bigg\|\sum_{i,j=1}^{m}\mathbf{D}g_i\big(\widetilde Y(t_n)\big)g_j\big(\widetilde Y(t_n)\big)I_{j,i}^{t_n,s}\Bigg\|^2\Bigg|\mathcal{F}_{t_n}\Bigg]\Bigg)\\
     \,\leq\,\,& 8 C_{Dg_i}^2\big(C_f^2+ \Citerated{1,R}\big)|s-t_n|^2.\numberthis\label{eq:G_11}  
\end{align*}
For $G_{1.2}$ in \eqref{eq:G_1_aim}, we apply \eqref{eq:Jen_int}, 
the Cauchy-Schwarz inequality, then using \eqref{eq:T01} in Lemma \ref{lmm:T01} with $k=2$ and \eqref{eq:||eror||^2n} in Lemma \ref{lmm:||eror||^2n} with $k=4$ we get
\begin{align*}
    G_{1.2} \,\leq\,\,& 4  \int_{0}^{1}\Big(\mathbb{E}\Big[\Big\|(1-\epsilon)\mathbf{D}^2g_i\Big(\widetilde Y(t_n)-\epsilon\big(  Y_{\theta}(s)-\widetilde Y(t_n)\big)\Big)\Big\|_{\mathbf{T_3}}^4  \Big|\mathcal{F}_{t_n}\Big]\Big)^{1/2}d\epsilon\\
    &\qquad\qquad\qquad\qquad\qquad\times \Big(\mathbb{E}\Big[\Big\|  Y_{\theta}(s)-\widetilde Y(t_n) \Big\|^8 \Big|\mathcal{F}_{t_n}\Big]\Big)^{1/2} \\
    \,\leq\,\,& 4\Ctaylor{2,R}^{1/2}  \Cscheme{4,R}^{1/2} \,|s-t_n|^2.\numberthis\label{eq:G_12}
\end{align*}
Substituting the bounds \eqref{eq:G_11} and \eqref{eq:G_12} back to  \eqref{eq:G_1_aim} before bringing together the terms in \eqref{eq:G_1}, we have
\begin{align*}
     \mathbb{E}\left[\left\|\Delta g_i\big(s,X(s),\widetilde Y(t_n)\big) \right\|^2\middle|\mathcal{F}_{t_n}\right] \leq  2\mathbb{E}\Big[\Big\|g_i(X(s))-g_i\big( Y_{\theta}(s)\big)\Big\|^2 \Big|\mathcal{F}_{t_n}\Big]  +C_{G }(R)  |s-t_n|^2,
\end{align*}
with $C_G(R) $ in \eqref{eq:C_Gr}. By bounding $\|g_i\|^2$ with $\|g\|_{\mathbf{F}(d\times m)}^2$, the statement of Lemma \ref{lmm:Gr} follows.
\end{proof}
The second lemma in the following gives the conditional second-moment bound of $E_{\theta}(s)$ as in \eqref{eq:defE(r)}, which is the first part of the one-step error in \eqref{eq:defE(r)_combined}.
\begin{lemma}  \label{lmm:CaseI}
Let $f$, $g$ satisfy Assumption \ref{ass:f+g} and \ref{ass:SDEmoments_power}. Let $X(s)$ be a solution of \eqref{eq:true} and 
\NOTE{$\widetilde E(s)$ be given by \eqref{eq:defE(r)_combined}}
with $E_{\theta}(s)$ defined in \eqref{eq:defE(r)}, with $s\in[t_n,t_{n+1}]$, $n\in\mathbb{N}$. 
In this case there exists a constant $C_E$ and an $\mathcal{F}_{t_n}$-measurable random variable $\overline C^{\{4(q+2)\}}_{M}$ such that 
\begin{multline*}
    \mathbb{E}\Big[\big\|  E_{\theta}(t_{n+1})\big\|^2 \Big|\mathcal{F}_{t_n}\Big] \leq\,\, \big\|\widetilde E(t_{n})\big\|^2+C_E(R)\int_{t_n}^{t_{n+1}} \mathbb{E}\left[\big\|E_{\theta}(r)\big\|^2 \middle|\mathcal{F}_{t_n}\right]dr\\
    +\overline C_M^{\{4(q+2)\}}(R)\,h_{n+1}^3, \quad a.s. \numberthis \label{eq:Case_I}
\end{multline*}
where 
\begin{align}
    C_E(R) :=\, 2K_1(R)+2c, \label{eq:C_E}
\end{align}
with constant $K_1$ as defined in \eqref{eq:K1}. The $\mathcal{F}_{t_n}$-measurable random variable $\overline C^{\{4(q+2)\}}_M$ is given by 
\begin{align*}
    \overline C_M^{\{4(q+2)\}}(R)\,:=\,m^4 C^2_{Df} C^2_{g_i}+2\overline K_2^{\{4(q+2)\}}+mC_{G}(R) , \numberthis \label{eq:C_h_bar}
\end{align*}
with the $\mathcal{F}_{t_n}$-measurable random variable $\overline K^{\{4(q+2)\}}_2$ in \eqref{eq:K_2_bar}, constant $C_{G}$ in Lemma \ref{lmm:Gr}.
Denote $\expect{\overline C^{\{4(q+2)\}}_M(R)}=:C_M(R)$, the finiteness of which is ensured in \eqref{eq:E[C_h]}.
\end{lemma}
We recall that the superscript notation in 
\eqref{eq:C_h_bar} follows the convention introduced in the statement of Lemma \ref{lmm:PR} and indicates the number of finite moments required of the SDE solution (see Remark \ref{rem:supScriNot}).
\begin{proof}
Throughout the proof, we restrict attention to trajectories on the event $\{h_{\min}<h_{n+1}\leq h_{\max}\}$, since by \eqref{eq:defE(r)}, $E_\theta(s)$ is only nonzero on this event, otherwise \eqref{eq:Case_I} holds trivially. Applying the \Note{stopping time variant of}
\ito formula \Note{(see Mao \& Yuan~\cite{mao2006stochastic})}
to \eqref{eq:defE(r)}, 
we have,
\begin{multline*}
\big\|E_{\theta}(t_{n+1})\big\|^2=\big\|\widetilde E(t_{n})\big\|^2+2\int_{t_n}^{t_{n+1}}\underbrace{\Big\langle E_{\theta}(r),\Delta f\big(X(r),\widetilde Y(t_n)\big) \Big\rangle}_{=:J_f} dr\\
+\sum_{i=1}^{m}\int_{t_n}^{t_{n+1}}\Note{\Big\|\underbrace{\Delta g_i\big(r,X(r),\widetilde Y(t_n)\big)}_{=:J_{g_i}}\Big\|^2} dr+2\sum_{i=1}^{m}\int_{t_n}^{t_{n+1}}\big\langle E_{\theta}(r),\Note{J_{g_i}}\big\rangle dW_i(r).
\end{multline*}
Take expectations on both sides conditional upon $\mathcal{F}_{t_n}$, and since
$\int_{t_n}^{t_{n+1}}\big|J_f\big| dr$ 
has finite expectation (by the boundedness of $\widetilde Y(t_n)$ in \eqref{eq:defYn} and the finiteness of absolute moments of $X(r)$ see \eqref{eq:SDEmoments}), using Fubini's Theorem (see for example \cite[Proposition 12.10]{Dineen2011}) and \eqref{eq:Wmoments1n2} we have, 
\begin{multline}
\qquad \mathbb{E}\Big[\big\|E_{\theta}(t_{n+1})\big\|^2 \Big|\mathcal{F}_{t_n}\Big] = \big\|\widetilde E(t_{n})\big\|^2 +2\int_{t_n}^{t_{n+1}}\mathbb{E}\big[J_f\big|\mathcal{F}_{t_n}\big]dr\\
+\sum_{i=1}^{m}\int_{t_n}^{t_{n+1}}\mathbb{E}\big[\Note{\|}J_{g_i}\Note{\|^2}\big|\mathcal{F}_{t_n}\big]dr,   \label{eq:err+ito}  
\end{multline}
By Lemma \ref{lmm:Gr}, we have the bound of $\Note{\|}J_{g_i}\Note{\|^2}$ in \eqref{eq:err+ito} as
\begin{align}
    \mathbb{E}\big[\Note{\|}J_{g_i}\Note{\|^2}\big|\mathcal{F}_{t_n}\big]
    \leq\, 2\mathbb{E}\Big[\big\|g(X(r))-g(Y_{\theta}(r))\big\|_{\mathbf{F}(d\times m)}^2 \Big|\mathcal{F}_{t_n}\Big] +C_{G}(R) |r-t_n|^2.  \label{eq:E[E_g]}
\end{align}
For $J_f$, by substituting $\Delta f$ with \eqref{eq:errorF} with adding in and subtracting out $f(Y_{\theta}(r))$, we have
\begin{align}
J_f=\Big\langle E_{\theta}(r),f(X(r))-f(Y_{\theta}(r))\Big\rangle +\underbrace{\Big\langle E_{\theta}(r),f(Y_{\theta}(r))-f\big(\widetilde Y(t_n)\big)\Big\rangle}_{=:H}. \label{eq:E_f}
\end{align}
Substituting \eqref{eq:E_f} and \eqref{eq:E[E_g]} back into \eqref{eq:err+ito}, we have
\begin{multline*}
\qquad \mathbb{E}\Big[\big\|E_{\theta}(t_{n+1})\big\|^2 \Big|\mathcal{F}_{t_n}\Big] \leq \big\|\widetilde E(t_{n})\big\|^2+mC_{G}(R) h_{n+1}^3 \\
+2\int_{t_n}^{t_{n+1}}\mathbb{E}\big[J_{f,g}\big|\mathcal{F}_{t_n}\big]dr+2\int_{t_n}^{t_{n+1}}\mathbb{E}\big[H\big|\mathcal{F}_{t_n}\big]dr, \numberthis \label{eq:err+H}
\end{multline*}
where 
\begin{align}
    J_{f,g}:=\Big\langle E_{\theta}(r),f(X(r))-f(Y_{\theta}(r))\Big\rangle+\big\|g(X(r))-g(Y_{\theta}(r))\big\|_{\mathbf{F}(d\times m)}^2. \label{eq:J_fg}
\end{align}
For $H$ in \eqref{eq:E_f}, and in a similar way to \eqref{eq:TaylorG}, we expand $f(Y_{\theta}(r))$ using Taylor's theorem around $\widetilde Y(t_n)$ to have
\begin{multline}
\qquad f(Y_{\theta}(r))-f\big(\widetilde Y(t_n)\big)=\mathbf{D}f\big(\widetilde Y(t_n)\big)\big(Y_{\theta}(r)-\widetilde Y(t_n)\big)\\+\int_{0}^{1}(1-\epsilon)\mathbf{D}^2f\Big(\widetilde Y(t_n)-\epsilon\cdot\big(Y_{\theta}(r)-\widetilde Y(t_n)\big)\Big)\Big[Y_{\theta}(r)-\widetilde Y(t_n)\Big]^2d\epsilon.  \label{eq:Rf}
\end{multline}
Then we substitute $Y_{\theta}(r)$ in the first term on the RHS of \eqref{eq:Rf} with \eqref{eq:deftheta} where we use the expanded form of the map as characterised in  \eqref{eq:Dgexpansion} for $s=r$.
Therefore, for the last term on the RHS of  \eqref{eq:err+H}, we have
\begin{align}
\mathbb{E}\big[H\big|\mathcal{F}_{t_n}\big]\,\leq\,  H_1+H_2+H_3+H_4+H_5+H_6, \label{eq:E[H]}
\end{align}
where
\begin{align*}
    H_1 :=&\,\,\mathbb{E}\Big[\Big<E_{\theta}(r)\boldsymbol{,}\,\, \mathbf{D}f\big(\widetilde Y(t_n)\big)|r-t_n|f\big(\widetilde Y(t_n)\big)\Big>\Big|\mathcal{F}_{t_n}\Big]; \\
    H_2 :=&\,\,\mathbb{E}\bigg[\bigg<E_{\theta}(r)\boldsymbol{,}\,\, \underbrace{\sum_{i=1}^{m}\mathbf{D}f\big(\widetilde Y(t_n)\big)g_i\big(\widetilde Y(t_n)\big)I_{i}^{t_n,r}}_{=:H_{2R}} \bigg>\bigg|\mathcal{F}_{t_n} \bigg];\\
    H_{3}:=&\,\,\mathbb{E}\Bigg[\Bigg\langle E_{\theta}(r),\,\, \frac{1}{2}\sum_{i=1}^{m}\mathbf{D}f\big(\widetilde Y(t_n)\big)\mathbf{D}g_i\big(\widetilde Y(t_n)\big)g_i\big(\widetilde Y(t_n)\big)\\
    &\qquad\qquad\qquad\qquad\qquad\qquad\times \left(\left(I_{i}^{t_n,r}\right)^2-|r-t_n|\right) \Bigg\rangle\Bigg|\mathcal{F}_{t_n} \Bigg];\\
    H_{4}:=&\,\,\mathbb{E}\Bigg[\Bigg\langle E_{\theta}(r),\,\,\frac{1}{2}\sum_{\substack{i,j=1\\i<j}}^{m}\mathbf{D}f\big(\widetilde Y(t_n)\big)\Big(\mathbf{D}g_i\big(\widetilde Y(t_n)\big)g_j\big(\widetilde Y(t_n)\big)\\
    &\quad\qquad\qquad\qquad+\mathbf{D}g_j\big(\widetilde Y(t_n)\big)g_i\big(\widetilde Y(t_n)\big)\Big) I_{i}^{t_n,r}I_{j}^{t_n,r} \Bigg\rangle\Bigg|\mathcal{F}_{t_n} \Bigg];\\
    H_{5}:=&\,\,\mathbb{E}\Bigg[ \Bigg\langle E_{\theta}(r),\,\, \sum_{\substack{i,j=1\\i<j}}^{m}\mathbf{D}f\big(\widetilde Y(t_n)\big)\Big(\mathbf{D}g_i\big(\widetilde Y(t_n)\big)g_j\big(\widetilde Y(t_n)\big)\\
    &\qquad\qquad\qquad\quad-\mathbf{D}g_j\big(\widetilde Y(t_n)\big)g_i\big(\widetilde Y(t_n)\big)\Big)A_{ij}(t_n, r)\Bigg\rangle\Bigg|\mathcal{F}_{t_n} \Bigg];\\
    H_{6}:=&\,\,\mathbb{E}\bigg[\bigg<E_{\theta}(r),\,\, \int_{0}^{1}(1-\epsilon)\mathbf{D}^2f\Big(\widetilde Y(t_n)-\epsilon\cdot\big(Y_{\theta}(r)-\widetilde Y(t_n)\big)\Big)\\
    &\qquad\qquad\qquad\qquad\qquad\qquad\quad\times\Big[Y_{\theta}(r)-\widetilde Y(t_n)\Big]^2d\epsilon\bigg>\bigg|\mathcal{F}_{t_n} \bigg].
\end{align*}
We will now determine suitable upper bounds for each of $H_1$, $H_2$, $H_3$, $H_4$, $H_5$, and $H_6$ in turn. For $H_1$ in \eqref{eq:E[H]}, by the Cauchy-Schwarz inequality, \eqref{eq:ab<a^2+b^2}, and \eqref{eq:constants define}, we have
\begin{align*}
H_1\leq&\,\,\mathbb{E}\Big[\|E_{\theta}(r)\|\, \big\|\mathbf{D}f\big(\widetilde Y(t_n)\big)\big\|_{\mathbf{F}}\,\big\|f\big(\widetilde Y(t_n)\big)\big\|\,|r-t_n|\,\Big|\mathcal{F}_{t_n}\Big]\\
\leq&\,\,\mathbb{E}\left[\frac{1}{2}\big\| \mathbf{D}f\big(\widetilde Y(t_n)\big)\big\|_{\mathbf{F}}^2\,\big\|f\big(\widetilde Y(t_n)\big)\big\|^2\|E_{\theta}(r)\|^2+\frac{1}{2}|r-t_n|^2\,\middle|\mathcal{F}_{t_n}\right]\\
\leq&\,\,\frac{1}{2} C_{Df}^2 C_{f}^2\,\,\mathbb{E}\Big[\|E_{\theta}(r)\|^2\Big|\mathcal{F}_{t_n}\Big]+\frac{1}{2}|r-t_n|^2.  \numberthis \label{eq:H_1}
\end{align*}
Next, for the analysis of  $H_2$ in \eqref{eq:E[H]}, by \eqref{eq:Wmoments1n2}, we firstly have
\begin{align}
    \mathbb{E}\big[H_{2R}\big|\mathcal{F}_{t_n} \big]=\sum_{i=1}^{m}\mathbf{D}f\big(\widetilde Y(t_n)\big)g_i\big(\widetilde Y(t_n)\big)\mathbb{E}\Big[I_{i}^{t_n,r}\Big|\mathcal{F}_{t_n} \Big] =0. \label{eq:H_2R_exp}
\end{align}
By \eqref{eq:Jen_sum}, the Cauchy-Schwarz inequality, \eqref{eq:constants define} and \eqref{eq:WvarCond} we also have 
\begin{align*}
    \mathbb{E}\Big[\big\|H_{2R}\big\|^2\Big|\mathcal{F}_{t_n} \Big] \leq\,\,& m\sum_{i=1}^{m}\big\|\mathbf{D}f\big(\widetilde Y(t_n)\big)\big\|_{\mathbf{F}}^2 \big\|g_i\big(\widetilde Y(t_n)\big)\big\|^2 \mathbb{E}\left[\left|I_{i}^{t_n,r} \right|^2 \middle|\mathcal{F}_{t_n} \right]\\
    \leq\,\,&m^2 C_{Df}^2 C_{g_i}^2|r-t_n|. \numberthis \label{eq:H_2R_var}
\end{align*}
Then, for $H_2$ in \eqref{eq:E[H]} we firstly expand $E_{\theta}(r)$ using \eqref{eq:defE(r)} to have
\begin{align*}
H_2=&\, \mathbb{E}\left[\left<\widetilde E(t_n)\boldsymbol{,}\,\, H_{2R} \right> \middle|\mathcal{F}_{t_n} \right] +\mathbb{E}\left[\left<\int_{t_n}^{r}\Delta f(X(p),Y(t_n))dp\boldsymbol{,}\,\, H_{2R}\right> \middle|\mathcal{F}_{t_n} \right]\\
&\quad\qquad+\mathbb{E}\left[\left<\sum_{i=1}^{m}\int_{t_n}^{r}\Delta g_i\big(p,X(p),\widetilde Y(t_n)\big)dW_{i}(p)\boldsymbol{,}\,\, H_{2R} \right> \middle|\mathcal{F}_{t_n} \right]\\
=:&\,\, H_{2.1}+H_{2.2}+H_{2.3}. \numberthis \label{eq:H_2}
\end{align*}
For $H_{2.1}$ in \eqref{eq:H_2}, by \eqref{eq:H_2R_exp} we have
\begin{align*}
    H_{2.1}=\left<\widetilde E(t_n)\boldsymbol{,}\,\, \mathbb{E}\big[H_{2R}\big|\mathcal{F}_{t_n} \big] \right> =0. \numberthis\label{eq:H_2.1}
\end{align*}
For $H_{2.2}$ in \eqref{eq:H_2}, by adding in and subtracting out $f(X(t_n))$ in $\Delta f$ in \eqref{eq:errorF}:
\begin{align*}
    H_{2.2}=&\,\, \mathbb{E}\Bigg[\Big<\int_{t_n}^{r} f(X(r))-f(X(t_n)) dp\boldsymbol{,}\,\, H_{2R} \Big> \Bigg|\mathcal{F}_{t_n} \Bigg]\\
    &\qquad+\mathbb{E}\Bigg[\Big<\int_{t_n}^{r} f(X(t_n))-f\big(\widetilde Y(t_n)\big) dp\boldsymbol{,}\,\, H_{2R}\Big> \Bigg|\mathcal{F}_{t_n} \Bigg]\\
    =:&\,\, H_{2.21}+H_{2.22}.\numberthis\label{eq:H_2.2}
\end{align*}
Similar to $H_{2.1}$ in \eqref{eq:H_2.1}, we have $H_{2.22} = 0$. For $H_{2.21}$ in \eqref{eq:H_2.2}, using the Cauchy-Schwarz inequality and \eqref{eq:H_2R_var} we have 
\begin{align*}
    H_{2.21}\leq&\,\, \mathbb{E}\left[\left\|\int_{t_n}^{r} f(X(p))-f(X(t_n)) dp\right\| \big\|H_{2R}\big\| \middle|\mathcal{F}_{t_n} \right]\\
    \leq& \Bigg(|r-t_n|\int_{t_n}^r\mathbb{E}\left[ \left\|f(X(p))-f(X(t_n)) \right\|^2 \middle|\mathcal{F}_{t_n} \right]dp\,\,\mathbb{E}\Big[\big\|H_{2R}\big\|^2\Big|\mathcal{F}_{t_n} \Big]\Bigg)^{1/2}\\
    \leq& \,\, m C_{Df} C_{g_i}|r-t_n|\left(\int_{t_n}^r\mathbb{E}\left[ \left\|f(X(p))-f(X(t_n)) \right\|^2\middle|\mathcal{F}_{t_n} \right]dp\right)^{1/2}. \numberthis\label{eq:H_2.21}
\end{align*}
By Taylor expansion of $f(X(p))$ around $f(X(t_n))$ to first order, and using \eqref{eq:Df+Dg}, the Cauchy-Schwarz inequality, Lemma \ref{lmm:PR} with $k=2$ and \eqref{eq:Jen_sum}, we have
\begin{align*}
     &\,\,\mathbb{E}\left[\big\|f(X(p))-f(X(t_n)) \big\|^2\middle|\mathcal{F}_{t_n} \right]\\
      =&\,\,\mathbb{E}\left[\left\|\int_{0}^{1}\mathbf{D}f(X(t_n)-\epsilon\cdot(X(p)-X(t_n))(X(p)-X(t_n))d\epsilon\right\|^2\middle|\mathcal{F}_{t_n} \right]\\
      \leq & \Big(\mathbb{E}\left[\|X(p)-X(t_n)\|^4\middle|\mathcal{F}_{t_n} \right] \Big)^{1/2}\\
     & \qquad\qquad\times\left(\mathbb{E}\left[\left\|\int_{0}^{1}\mathbf{D}f\big(X(t_n)-\epsilon\cdot(X(p)-X(t_n)\big)d\epsilon\right\|_{\mathbf{F}}^4\middle|\mathcal{F}_{t_n} \right] \right)^{1/2}\\
     \leq& \,\,\overline C^{\{4(q+2)\}}_{H2.21}|p-t_n|, \numberthis\label{eq:H_2.21_PR}
\end{align*}
where 
\begin{align}
     \overline C^{\{4(q+2)\}}_{H2.21}:= \,\,c_3^2\left(\CpathBar{4(q+2)}{2,X(p)} \left(1+3^{4q_1+4}\mathbb{E} \Big[\sup_{p\in[t_n,t_{n+1}]}\|X(p)\|^{4q_1+4}\Big|\mathcal{F}_{t_n} \Big]\right)\right)^{1/2} .   \label{eq:C_H221_bar}
\end{align}
Substituting \eqref{eq:H_2.21_PR} back to \eqref{eq:H_2.21} and using that $H_{2.22}=0$, we have
\begin{align*}
    H_{2.2} \leq\,\,& m C_{Df} C_{g_i} \left(\overline  C^{\{4(q+2)\}}_{H2.21}\right)^{1/2}|r-t_n|^2. \numberthis \label{eq:H_2.2_end}
\end{align*}
For $H_{2.3}$ as in \eqref{eq:H_2}, by Cauchy-Schwarz inequality, \eqref{eq:Jen_sum}, \eqref{eq:constants define}, \eqref{eq:WvarCond} and \ito's isometry we have    
\begin{align*}
    H_{2.3}\leq& \left(\mathbb{E}\left[\left\|\sum_{i=1}^{m}\int_{t_n}^{r}\Delta g_i\big(p,X(p),\widetilde Y(t_n)\big)dW_{i}(p)\right\|^2\middle|\mathcal{F}_{t_n} \right]\right)^{1/2}\\
    &\qquad\qquad\quad\times \left(\mathbb{E}\left[\left\|\sum_{i=1}^{m}\mathbf{D}f\big(\widetilde Y(t_n)\big)g_i\big(\widetilde Y(t_n)\big)I_{i}^{t_n,r} \right\|^2 \middle|\mathcal{F}_{t_n} \right]\right)^{1/2}\\
    \leq& \left(m\sum_{i=1}^{m}\int_{t_n}^{r}\mathbb{E}\left[\left\|\Delta g_i\big(p,X(p),\widetilde Y(t_n)\big) \right\|^2\middle|\mathcal{F}_{t_n} \right]dp\right)^{1/2}
    m C_{Df} C_{g_i} |r-t_n|^{1/2}.
\end{align*}
Then, by Lemma \ref{lmm:Gr} we have 
\begin{multline*}
    \qquad H_{2.3} \leq\bigg(2m^2\int_{t_n}^{r} \mathbb{E}\Big[\big\|g(X(p))-g(Y_{\theta}(p))\big\|_{\mathbf{F}(d\times m)}^2 \Big|\mathcal{F}_{t_n}\Big]dp \\+C_{G}(R)  |r-t_n|^3\bigg)^{1/2}m C_{Df} C_{g_i} |r-t_n|^{1/2}.
\end{multline*}
Since the integrand $\mathbb{E}\Big[\big\|g(X(p))-g(Y_{\theta}(p))\big\|_{\mathbf{F}(d\times m)}^2 \Big|\mathcal{F}_{t_n}\Big]$ is non-negative for all $p\in[t_n,t_{n+1}]$, we can replace the upper limit of integration with $t_{n+1}$.
With $\sqrt{a+b}\leq\sqrt{a}+\sqrt{b}$, we have 
\begin{multline}
    H_{2.3}
    \leq \sqrt{2}m^2 C_{Df} C_{g_i}|r-t_n|^{1/2} \left(\int_{t_n}^{t_{n+1}} \mathbb{E}\Big[\big\|g(X(r))-g(Y_{\theta}(r))\big\|_{\mathbf{F}(d\times m)}^2 \Big|\mathcal{F}_{t_n}\Big]dr\right)^{1/2}\\
    +m C_{Df} C_{g_i}C_{G}(R)^{1/2} |r-t_n|^{2}. \label{eq:H_2.3} 
\end{multline}
Notice that we changed the variable of integration from $p$ back to $r$ for consistency. Substituting  \eqref{eq:H_2.1}, 
\eqref{eq:H_2.2_end} and \eqref{eq:H_2.3} back into \eqref{eq:H_2}, we have
\begin{multline}
    H_{2} \leq \,\,m C_{Df} C_{g_i} \Bigg(\left(\overline C^{\{4(q+2)\}}_{H2.21}\right)^{1/2}+C_{G}(R)^{1/2} \Bigg)|r-t_n|^2\\
    +\sqrt{2}m^2 C_{Df} C_{g_i}|r-t_n|^{1/2} \left(\int_{t_n}^{t_{n+1}} \mathbb{E}\Big[\big\|g(X(r))-g(Y_{\theta}(r))\big\|_{\mathbf{F}(d\times m)}^2 \Big|\mathcal{F}_{t_n}\Big]dr\right)^{1/2}.   \label{eq:H_2_end}
\end{multline}
For $H_{3}$ in \eqref{eq:E[H]}, by the Cauchy-Schwarz inequality, triangle inequality,  \eqref{eq:Jen_sum}, \eqref{eq:ab<a^2+b^2}, \eqref{eq:even Gaussian_p}, \eqref{eq:Df+Dg} and \eqref{eq:constants define}, we have 
\begin{align*}
H_{3}\leq\,\,&\mathbb{E}\Bigg[\frac{1}{4}\sum_{i=1}^{m}\Bigg(  \big\|\mathbf{D}f\big(\widetilde Y(t_n)\big)\big\|_{\mathbf{F}}^2\,\big\|\mathbf{D}g_i\big(\widetilde Y(t_n)\big)\big\|_{\mathbf{F}}^2\,\big\|g_i\big(\widetilde Y(t_n)\big)\big\|^2\|E_{\theta}(r)\|^2 \\
&\qquad\qquad\qquad\qquad\qquad\qquad\qquad + 2\left|I_{i}^{t_n,r}\right|^4+ 2|r-t_n|^2\Bigg) \Bigg|\mathcal{F}_{t_n} \Bigg]\\
\leq\,\,&\frac{m}{4} C_{Df}^2 C_{Dg_i}^2  C_{g_i}^2\,\,\mathbb{E}\Big[\|E_{\theta}(r)\|^2\Big|\mathcal{F}_{t_n}\Big]+ \frac{(\Cp{4}+1)m}{2}|r-t_n|^2. \numberthis\label{eq:H_3}
\end{align*}
For $H_{4}$ in \eqref{eq:E[H]}, by the Cauchy-Schwarz inequality, conditional independence of the \ito integrals, \eqref{eq:WvarCond}, triangle inequality, \eqref{eq:ab<a^2+b^2}, \ito's isometry, \eqref{eq:Df+Dg}, and \eqref{eq:constants define}, we have 
\begin{align*}
H_{4}\leq\,\,&\mathbb{E}\Bigg[\frac{1}{2}\sum_{\substack{i,j=1\\i<j}}^{m}\|E_{\theta}(r)\| \big\|\mathbf{D}f\big(\widetilde Y(t_n)\big)\big\|_{\mathbf{F}}\Big(\big\|\mathbf{D}g_i\big(\widetilde Y(t_n)\big)\big\|_{\mathbf{F}}\big\|g_j\big(\widetilde Y(t_n)\big)\big\|\\
&\qquad\qquad\qquad+\big\|\mathbf{D}g_j\big(\widetilde Y(t_n)\big)\big\|_{\mathbf{F}}\big\|g_i\big(\widetilde Y(t_n)\big)\big\|\Big)
\left|I_{i}^{t_n,r}\right|\,\left|I_{j}^{t_n,r}\right|  \Bigg| \mathcal{F}_{t_n} \Bigg]\\
\leq\,\,&\frac{1}{4}m(m-1) C_{Df}^2 C_{Dg_i}^2  C_{g_i}^2\mathbb{E}\Big[\|E_{\theta}(r)\|^2\Big|\mathcal{F}_{t_n}\Big]+\frac{1}{8}m(m-1)|r-t_n|^2. \numberthis\label{eq:H_4}
\end{align*}
For $H_{5}$ in \eqref{eq:E[H]}, by the Cauchy-Schwarz inequality, triangle  inequality, \eqref{eq:ab<a^2+b^2}, \eqref{eq:constants define}, \eqref{eq:Df+Dg}, and Lemma \ref{lem:levy bound} with $b=2$, we have 
\begin{align*}
H_{5}\leq&\,\,\mathbb{E}\Bigg[\frac{1}{2}\sum_{\substack{i,j=1\\i<j}}^{m}\Bigg(\big\|\mathbf{D}f\big(\widetilde Y(t_n)\big)\|_{\mathbf{F}}^2\|E_{\theta}(r)\|^2\Big(\big\|\mathbf{D}g_i\big(\widetilde Y(t_n)\big)\big\|_{\mathbf{F}}^2\big\|g_j\big(\widetilde Y(t_n)\big)\big\|^2\\
&\quad\qquad\qquad+\big\|\mathbf{D}g_i\big(\widetilde Y(t_n)\big)\big\|_{\mathbf{F}}^2\big\|g_j\big(\widetilde Y(t_n)\big)\big\|^2\Big)+\big|A_{ij}(t_n, r)\big|^2\Bigg)\Bigg|\mathcal{F}_{t_n} \Bigg]\\
\leq&\,\,\frac{1}{2}m(m-1)  C_{Df}^2 C_{Dg_i}^2  C_{g_i}^2\mathbb{E}\Big[\|E_{\theta}(r)\|^2\Big|\mathcal{F}_{t_n}\Big]\\
&\quad\qquad\qquad+\frac{1}{4}m(m-1)(\Clevy{2})^2|r-t_n|^2. \numberthis\label{eq:H_5}
\end{align*}
For $H_{6}$ in \eqref{eq:E[H]}, by the Cauchy-Schwarz inequality, triangle inequality, and \eqref{eq:ab<a^2+b^2}, we have (noting that $\| [\cdot ]^2 \|_\mathbf{F}=\| \cdot \|^2$)
\begin{align*}
H_{6}\leq\,\,& \mathbb{E}\bigg[\|E_{\theta}(r)\| \big\| Y_{\theta}(r)-\widetilde Y(t_n)\big\|^2\\
&\qquad\qquad\times\left\| \int_{0}^{1}(1-\epsilon)\mathbf{D}^2f\Big(\widetilde Y(t_n)-\epsilon\cdot\big(Y_{\theta}(r)-\widetilde Y(t_n)\big)\Big)d\epsilon \right\|_{\mathbf{T}_3}\bigg|\mathcal{F}_{t_n} \bigg]\\
\leq\,\,& \frac{1}{2}\mathbb{E}\Big[\|E_{\theta}(r)\|^2\Big|\mathcal{F}_{t_n}\Big]+ \frac{1}{2}\underbrace{\sqrt{\mathbb{E}\left[\big\| Y_{\theta}(r)-\widetilde Y(t_n)\big\|^8\middle|\mathcal{F}_{t_n} \right]}}_{H_{6.1}}\\
&\qquad\times\underbrace{\sqrt{\mathbb{E}\left[ \left\| \int_{0}^{1}(1-\epsilon)\mathbf{D}^2f\Big(\widetilde Y(t_n)-\epsilon\cdot\big(Y_{\theta}(r)-\widetilde Y(t_n)\big)\Big)d\epsilon \right\|_{\mathbf{T}_3}^4 \middle|\mathcal{F}_{t_n} \right]}}_{H_{6.2}}.
\end{align*}
From \eqref{eq:||eror||^2n} in Lemma \ref{lmm:||eror||^2n} with $k=4$, we have   $H_{6.1}\leq \Cscheme{4,R}^{1/2} |r-t_n|^2.$
From \eqref{eq:T01} in Lemma \ref{lmm:T01} with $k=2$, we have $H_{6.2} \leq  \Ctaylor{2,R}^{1/2}$.  
Therefore, $H_{6}$ in \eqref{eq:E[H]} becomes  
\begin{align}
H_{6}\leq \, \frac{1}{2}\mathbb{E}\Big[\|E_{\theta}(r)\|^2\Big|\mathcal{F}_{t_n}\Big] +\frac{1}{2}  \Cscheme{4,R}^{1/2}\, \Ctaylor{2,R}^{1/2}\,|r-t_n|^2.   \label{eq:H_6}
\end{align}
Substituting \eqref{eq:H_1}, \eqref{eq:H_2_end}, \eqref{eq:H_3}, \eqref{eq:H_4}, \eqref{eq:H_5} and \eqref{eq:H_6} back into \eqref{eq:E[H]} for $H$, we have  
\begin{multline}
    \mathbb{E}\big[H \big|\mathcal{F}_{t_n}\big] \leq \,\, K_1(R) \mathbb{E}\Big[\|E_{\theta}(r)\|^2\Big|\mathcal{F}_{t_n}\Big]+\overline K^{\{4(q+2)\}}_2(R)|r-t_n|^2\\
    +\sqrt{2}m^2 C_{Df} C_{g_i}|r-t_n|^{1/2}\left(\int_{t_n}^{t_{n+1}} \mathbb{E}\Big[\big\|g(X(r))-g(Y_{\theta}(r))\big\|_{\mathbf{F}(d\times m)}^2 \Big|\mathcal{F}_{t_n}\Big]dr\right)^{1/2},  \label{eq:E[H]_end}
\end{multline}
where  
\begin{align}
K_1(R) :=\,\,\frac{1}{2} + \frac{1}{2} C_{Df}^2 C_{f}^2+m(m-1) C_{Df}^2 C_{Dg_i}^2 C_{g_i}^2,\label{eq:K1} 
\end{align}
and with $\overline C^{\{4(q+2)\}}_{H2.21}$ from \eqref{eq:C_H221_bar}
\begin{multline}
\overline K^{\{4(q+2)\}}_2(R) :=\,\,\frac{1}{2}+m C_{Df} C_{g_i} \bigg(\left(\overline C^{\{4(q+2)\}}_{H2.21}\right)^{1/2}
+\frac{1}{2}(\Cp{4}+1)m\\
+\frac{1}{4}m(m-1)\big(1+\left(\Clevy{2}\right)^2\big)+\frac{1}{2}  \Cscheme{4,R}^{1/2} \Ctaylor{2,R}^{1/2}.   \label{eq:K_2_bar}
\end{multline}
Substituting $\mathbb{E}[H|\mathcal{F}_{t_n}]$ from \eqref{eq:E[H]_end} back into \eqref{eq:err+H}, we have  
\begin{multline}
    \mathbb{E}\Big[\big\|E_{\theta}(t_{n+1})\big\|^2 \Big|\mathcal{F}_{t_n}\Big]\leq\big\|\widetilde E(t_n)\big\|^2+2K_1(R)\int_{t_n}^{t_{n+1}}\mathbb{E}\Big[\|E_{\theta}(r)\|^2\Big|\mathcal{F}_{t_n}\Big]dr\\
    +mC_{G}(R) h_{n+1}^3 +\overline K^{\{4(q+2)\}}_2(R) h_{n+1}^3\\
    +2\int_{t_n}^{t_{n+1}}\mathbb{E}\big[J_{f,g}\big|\mathcal{F}_{t_n}\big]dr+\sqrt{2 }m^2 C_{Df} C_{g_i}h_{n+1}^{3/2}\\
    \times\left(\int_{t_n}^{t_{n+1}} \mathbb{E}\Big[\big\|g(X(r))-g(Y_{\theta}(r))\big\|_{\mathbf{F}(d\times m)}^2 \Big|\mathcal{F}_{t_n}\Big]dr\right)^{1/2}. \label{eq:err_combined}
\end{multline}
Using \eqref{eq:ab<a^2+b^2} on the last term on the RHS of \eqref{eq:err_combined}, we have 
\begin{multline}
    \quad\mathbb{E}\Big[\big\|E_{\theta}(t_{n+1})\big\|^2 \Big|\mathcal{F}_{t_n}\Big]\leq\,\, \big\|\widetilde E(t_n)\big\|^2+2K_1(R)\int_{t_n}^{t_{n+1}} \mathbb{E}\Big[\|E_{\theta}(r)\|^2\Big|\mathcal{F}_{t_n}\Big]dr\\
    +\overline C_M^{\{4(q+2)\}}(R)\,h_{n+1}^3
    +2\int_{t_n}^{t_{n+1}}\mathbb{E}\bigg[J_{f,g}+\frac{1}{2}\Big\|g(X(r))-g(Y_{\theta}(r))\big\|_{\mathbf{F}(d\times m)}^2\bigg|\mathcal{F}_{t_n}\bigg]dr,   \label{eq:err_combined2}
\end{multline}
where $\overline C^{\{4(q+2)\}}_M$ is as defined in \eqref{eq:C_h_bar}.
Recall $J_{f,g}$ is given in \eqref{eq:J_fg} so that
\begin{multline*}
    J_{f,g}+\frac{1}{2}\Big\|g(X(r))-g(Y_{\theta}(r))\big\|_{\mathbf{F}(d\times m)}^2\\ =\Big\langle E_{\theta}(r),f(X(r))-f(Y_{\theta}(r))\Big\rangle+\frac{3}{2}\big\|g(X(r))-g(Y_{\theta}(r))\big\|_{\mathbf{F}(d\times m)}^2.
\end{multline*}
By Assumption \ref{ass:SDEmoments_power} we can apply the monotone condition \eqref{eq:monotone}:  
\begin{align*}
    \mathbb{E}\Big[\big\|E_{\theta}(t_{n+1})\big\|^2 \Big|\mathcal{F}_{t_n}\Big] \leq\,\, \big\|\widetilde E(t_n)\big\|^2+C_E(R)\int_{t_n}^{t_{n+1}} \mathbb{E}\Big[\|E_{\theta}(r)\|^2\Big|\mathcal{F}_{t_n}\Big]dr+\overline C_M^{\{4(q+2)\}}(R)\,h_{n+1}^3, 
\end{align*}
where $C_E(R)$ is in \eqref{eq:C_E}.

To obtain the the final estimate on 
$C_M(R)$ in the Lemma we use the 
explicit form of $\overline C_M^{\{4(q+2)\}}$, $\overline K_2^{\{4(q+2)\}}$ , $\overline C_{H2.21}^{\{4(q+2)\}}$, given by \eqref{eq:C_h_bar}, \eqref{eq:K_2_bar}, and \eqref{eq:C_H221_bar} respectively, \eqref{eq:E[C_PR_bar]} in the statement of Lemma \ref{lmm:T01}, \eqref{eq:SDEmoments}, and Assumption \ref{ass:SDEmoments_power} we bound the expectation of $\overline C_M^{\{4(q+2)\}}$ as follows,
\begin{align*}
       C_M(R)\,:=\,\,&\mathbb{E}\Bigg[ \overline C_M^{\{4(q+2)\}}(R)\Bigg]\\
      \leq&\,m^4 C^2_{Df} C^2_{g_i}+  2m C_{Df} C_{g_i} \Big(  c_3 \Cpath{2}^{1/4} \big(1+3^{q_1+1}\Cx\big)+C_{G}(R)^{1/2} \Big)\\
      &\quad+\frac{1}{2}(\Cp{4}+1)m+\frac{1}{2}m(m-1)\big(1+(\Clevy{2})^2\big)\\
      &\quad+  \Cscheme{4,R}^{1/2} \Ctaylor{2,R}^{1/2}+mC_{G}(R)+1. \numberthis \label{eq:E[C_h]}
\end{align*}
\end{proof}

\subsection{Proof of Theorem \ref{thm:result} on strong convergence.} \label{sec:proof_thm_4.1}
\begin{proof}
Firstly, by \eqref{eq:defE(r)_combined} we have the conditional second-moment bound of the one-step error as 
\begin{align}
\mathbb{E}\Big[\big\|\widetilde E(t_{n+1})\big\|^2 \Big|\mathcal{F}_{t_n}\Big] =\,\,\mathbb{E}\Big[\big\|E_{\theta}(t_{n+1})\big\|^2 \Big|\mathcal{F}_{t_n}\Big]+\mathbb{E}\Big[\big\|E_{\varphi}(t_{n+1})\big\|^2 \Big|\mathcal{F}_{t_n}\Big],  \label{eq:err_combined_Ft}
\end{align}
where by \eqref{eq:defbackstop} and \eqref{eq:defE_varphi}, the one-step error bound of the backstop map yields 
\begin{align}
    \mathbb{E}\Big[\big\|E_{\varphi}(t_{n+1})\big\|^2 \Big|\mathcal{F}_{t_n}\Big] \leq\,\, \big\|\widetilde E(t_n)\big\|^2+C_{B_1}\int_{t_n}^{t_{n+1}} \mathbb{E}\Big[\|E_{\varphi}(r)\|^2\Big|\mathcal{F}_{t_n}\Big]dr+C_{B_2}\,h_{n+1}^3, \quad a.s.   \label{eq:varphi_error}
\end{align}
Therefore, by substituting \eqref{eq:Case_I} and \eqref{eq:varphi_error} into \eqref{eq:err_combined_Ft}, and recalling \eqref{eq:defE(r)_combined} we have for any $h_{n+1}$ that satisfies Assumption \ref{ass:h},
\begin{multline*}
\mathbb{E}\Big[\big\|\widetilde E(t_{n+1})\big\|^2 \Big|\mathcal{F}_{t_n}\Big]
\leq\,\, \big\|\widetilde E(t_{n})\big\|^2  + \Gamma_1(R) \int_{t_n}^{t_{n+1}} \mathbb{E}\Big[\big\|\widetilde E(r)\big\|^2 \Big|\mathcal{F}_{t_n}\Big]dr\\
+\overline \Gamma_2^{\{4(q+2)\}}\big(R \big)   h_{n+1} ^3,\quad a.s. \numberthis \label{eq:err_before_sum}
\end{multline*}
where we define $\Gamma_1$, $\overline \Gamma_2$ and by \eqref{eq:E[C_h]} its expected form as 
\begin{align*}
    \Gamma_1(R):=&C_E(R)+C_{B_1};\\
    \overline \Gamma_2^{\{4(q+2)\}}\big(R\big) :=&\overline C_M^{\{4(q+2)\}}(R)+ C_{B_2};\\
    \Gamma_2(R):=&\mathbb{E}\bigg[\overline \Gamma_2^{\{4(q+2)\}}\big(R \big) \bigg] \leq C_M(R) + C_{B_2}. \numberthis \label{eq:E[Gamma_2]}
\end{align*}
For a fixed $t>0$, let $N^{(t)}$ be as in Definition \ref{def:N}, we multiply both sides of \eqref{eq:err_before_sum} with the indicator function $\mathbf{1}_{\{N^{(t)}> n+1\}}$ and sum  up the steps excluding the last step $N^{(t)}$ to have 
\begin{multline}
   \sum_{n=0}^{N^{(t)}-2}\mathbb{E}\Big[\big\|\widetilde E(t_{n+1})\big\|^2 \Big|\mathcal{F}_{t_n}\Big] \mathbf{1}_{\{N^{(t)}> n+1\}} \leq \sum_{n=0}^{N^{(t)}-2}\big\|\widetilde E(t_{n})\big\|^2\mathbf{1}_{\{N^{(t)}> n+1\}} \\
   +\, \Gamma_1(R) \sum_{n=0}^{N^{(t)}-2}\int_{t_n}^{t_{n+1}}\mathbb{E}\Big[\big\|\widetilde E(r)\big\|^2 \Big|\mathcal{F}_{t_n}\Big]\mathbf{1}_{\{N^{(t)}> n+1\}}dr\\
   \qquad+ \overline \Gamma_2^{\{4(q+2)\}}(R)\sum_{n=0}^{N^{(t)}-2} h_{n+1}^3\mathbf{1}_{\{N^{(t)}> n+1\}}.  \label{eq:err_summed}
\end{multline}
Since $t\in\big[t_{N^{(t)}-1},t_{N^{(t)}}\big]$, we use \eqref{eq:err_before_sum} to express the last step, noting that it holds when $t_n,t_{n+1}$ are replaced by $t_{N^{(t)}-1}$ and \Note{$t$} respectively:
\begin{multline}
    \mathbb{E}\Big[\big\|\widetilde E(t )\big\|^2 \Big|\mathcal{F}_{t_{N^{(t)}-1}}\Big] \leq\, \big\|\widetilde E(t_{N^{(t)}-1})\big\|^2+\Gamma_1(R)\int_{t_{N^{(t)}-1}}^{t} \mathbb{E}\Big[\big\|\widetilde E(r)\big\|^2 \Big|\mathcal{F}_{t_{N^{(t)}-1}}\Big]dr\\
    +\overline \Gamma_2^{\{4(q+2)\}}(R) \big|t-t_{N^{(t)}-1}\big|^3.   \label{eq:err_extra}
\end{multline}
By adding the both sides of \eqref{eq:err_summed} and \eqref{eq:err_extra},  
and taking an expectation: 
\begin{align*}
    &\left. \begin{array}{l}
        \qquad\mathbb{E}\Bigg[\displaystyle \sum_{n=0}^{N^{(t)}-2}\Big(\mathbb{E}\Big[\big\|\widetilde E(t_{n+1})\big\|^2 \Big|\mathcal{F}_{t_n}\Big]-\big\|\widetilde E(t_{n})\big\|^2\Big)\mathbf{1}_{\{N^{(t)}> n+1\}} \\
    \,\,\quad\quad\qquad\quad\qquad\qquad+\mathbb{E}\Big[\big\|\widetilde E(t)\big\|^2 \Big|\mathcal{F}_{t_{N^{(t)}-1}}\Big]-\big\|\widetilde E(t_{N^{(t)}-1})\big\|^2 \Bigg] \\
    \end{array}\right\} =:\text{LHS}\\
    &\left. \begin{array}{l}
        \leq\,\,\Gamma_1(R)\mathbb{E}\Bigg[ \displaystyle \sum_{n=0}^{N^{(t)}-2}\int_{t_n}^{t_{n+1}}\mathbb{E}\Big[\big\| \widetilde E(r)\big\|^2 \Big|\mathcal{F}_{t_n}\Big]\mathbf{1}_{\{N^{(t)}> n+1\}}dr\\
    \qquad\qquad\qquad\qquad\qquad\qquad+\displaystyle \int_{t_{N^{(t)}-1}}^{t} \mathbb{E}\Big[\big\| \widetilde E(r)\big\|^2 \Big|\mathcal{F}_{t_{N^{(t)}-1}}\Big]dr\Bigg] \\
    \end{array}\right\} =: \text{R}_1  \\
    &\left. \begin{array}{l}
        \quad+\mathbb{E}\Bigg[\overline \Gamma_2^{\{4(q+2)\}}(R)\Bigg(\displaystyle\sum_{n=0}^{N^{(t)}-2} h_{n+1}^3\mathbf{1}_{\{N^{(t)}> n+1\}}+\big|t-t_{N^{(t)}-1}\big|^3\Bigg)\Bigg]
    \end{array}\right\} =: \text{R}_2 \numberthis \label{eq:LHS=R1+R2}
\end{align*}
where we analyse \eqref{eq:LHS=R1+R2} ($\text{LHS}\leq\text{R}_1+\text{R}_2$) below. 
\NOTE{For the LHS in (7.51), $N^{(t)}$ is a random number taking value from $N^{(t)}_{\min}$ to $N^{(t)}_{\max}$, and $\mathbf{1}_{\{N^{(t)}> n+1\}}$ is an
$\mathcal{F}_{t_{n}}$-measurable random variable. Therefore it is useful decompose the range of $n$ into three parts on each trajectory.}
First, when $n<N^{(t)}-1$, then $1_{\{N^{(t)}>n+1\}}=1_{\{N^{(t)}>n\}}=1$. Second,  when $n=N^{(t)}-1$, then $1_{\{N^{(t)}>n+1\}}=0$ and $1_{\{N^{(t)}>n\}}=1$. Finally, when $n>N^{(t)}-1$, then $1_{\{N^{(t)}>n+1\}}=1_{\{N^{(t)}>n\}}=0$. 
Hence we obtain a telescoping sum with the appropriate cancellation that terminates at $\mathbb{E}\big[\|\widetilde{E}(t_{N^{(t)}-1})\|^2\,1_{\{N^{(t)}>N^{(t)}-1\}}\big]=\mathbb{E}\big[\|\widetilde{E}(t_{N^{(t)}-1})\|^2\big]$.
Applying this with the tower property for conditional expectations, and using the fact that $\|\widetilde{E}(t_0)\|^2=0$, we have
\begin{align*}
    \text{LHS}= &\sum_{n=0}^{N_{\max}^{(t)}-2}\mathbb{E}\Big[\big\|\widetilde E(t_{n+1})\big\|^2 \mathbf{1}_{\{N^{(t)}> n+1\}} -\big\|\widetilde E(t_{n})\big\|^2 \mathbf{1}_{\{N^{(t)}> n+1\}}\Big]\\
    &\qquad\qquad\qquad\qquad\qquad+\mathbb{E}\Big[\mathbb{E}\Big[\big\|\widetilde E(t )\big\|^2 \Big|\mathcal{F}_{t_{N^{(t)}-1}}\Big]-\big\|\widetilde E(t_{N^{(t)}-1})\big\|^2\Big] \\
    =\,\,&\mathbb{E}\Big[\big\|\widetilde E(t_{N^{(t)}-1})\big\|^2\Big]-\mathbb{E}\Big[\big\|\widetilde E(t_{0})\big\|^2\Big]+ \mathbb{E}\big[\big\|\widetilde E(t )\big\|^2  \big]-\mathbb{E}\Big[\big\|\widetilde E(t_{N^{(t)}-1})\big\|^2\Big]\\
    =\,\,&\mathbb{E}\Big[\big\|\widetilde E(t )\big\|^2  \Big]. \numberthis \label{eq:LHS}
\end{align*}
\NOTE{Consider the term  $\text{R}_1$ on the RHS of \eqref{eq:LHS=R1+R2}. By Definition \ref{def:N} we have each $n=N^{(r)}-1$ for $r\in[t_n,t_{n+1}]$. So we restate  $\mathcal{F}_{t_n}$ as $\mathcal{F}_{t_{N^{(r)}-1}}$, and the indicator function as $\mathbf{1}_{\{N^{(t)}>N^{(r)}\}}$. Summing up all the steps results in an integral from $0$ to $t_{N^{(t)}-1}$ that}
\begin{align*}
    \text{R}_1=\,\,&\Gamma_1(R)\mathbb{E}\Bigg[ \int_{0}^{t_{N^{(t)}-1}}\mathbb{E}\Big[\big\| \widetilde E(r)\big\|^2\mathbf{1}_{\{N^{(t)}>N^{(r)}\}} \Big|\mathcal{F}_{t_{N^{(r)}-1}}\Big]dr\\
    &\qquad\qquad\qquad\qquad\qquad\qquad+\int_{t_{N^{(t)}-1}}^{t} \mathbb{E}\Big[\big\| \widetilde E(r)\big\|^2 \Big|\mathcal{F}_{t_{N^{(t)}-1}}\Big]dr\Bigg]\\
    \leq\,\,&\Gamma_1(R)\int_{0}^{t} \mathbb{E}\Big[\big\| \widetilde E(r)\big\|^2 \Big]dr. \numberthis \label{eq:R_1}
\end{align*}
For $\text{R}_2$ in \eqref{eq:LHS=R1+R2}, by \eqref{eq:E[Gamma_2]}, Definition \ref{def:N} and $\rho\hmin=\hmax$, we have
\begin{equation}
  \text{R}_2\leq\Gamma_2(R) N^{(t)}_{\max}\hmax^3
    \leq \Gamma_2(R)\left(\rho t+1\right)\hmax^2.  \label{eq:R_2}
\end{equation}
We see that $4(q+2)$ is the minimum number of finite SDE moments required for a finite \NOTE{$\text{R}_2$}, and this is guaranteed by Assumption \ref{ass:SDEmoments_power}. 
Combining \eqref{eq:LHS}, \eqref{eq:R_1}  and  \eqref{eq:R_2} back into \eqref{eq:LHS=R1+R2}, for all $t\in [0,T]$, we have
\begin{align*}
    \mathbb{E}\Big[\big\| \widetilde E(t)\big\|^2 \Big]\,\,\leq\,\, \Gamma_1(R)\int_{0}^{t} \mathbb{E}\Big[\big\| \widetilde E(r)\big\|^2 \Big]dr +\Gamma_2(R)\left(\rho t+1\right)\hmax^2.
\end{align*}
By Gronwall's inequality (see \cite[Thm. 8.1]{mao2007SDEapp}), we have for all $t\in [0,T]$
\begin{align}
   \Big(\mathbb{E}\Big[\big\| \widetilde E(t)\big\|^2 \Big]\Big)^{\frac{1}{2}}\leq C(R,\rho,t)\,\hmax. \label{eq:result_t}
\end{align}
Taking the maximum over $t$ on the both sides, the proof follows with
\begin{align*}
    C(R,\rho,t):=\sqrt{\big(  C_M(R)+ C_{B_2}\big)\left(\rho t+1\right)\exp\Big(t\big(C_E(R)+C_{B_1}\big)\Big)}.
\end{align*}
\end{proof}

\subsection{Proof of Theorem \ref{thrm:MarkovIneq} on the probability of using the backstop.} \label{sec:proof_thm_4.2}
\begin{proof}
By \eqref{eq:defh} and by the Markov inequality we have
\begin{align}
\mathbb{P}\big[h_{n+1}= h_{\min}\big]=\mathbb{P}\left[\frac{h_{\max}}{\big\|\widetilde Y(t_n)\big\|^{1/\kappa}}\leq h_{\min}\right]\leq\frac{\mathbb{E}\Big[\big\|\widetilde Y(t_n)\big\|^2\Big]}{\rho^{2\kappa}}. \label{eq:prob}
\end{align}
By adding in and subtracting out $X(t_n)$ together with the tower property of conditional expectation, \eqref{eq:Jen_sum}, \eqref{eq:defE(r)_combined} and \eqref{eq:SDEmoments}, we have
\begin{align*}
    \mathbb{E}\Big[\big\|\widetilde Y(t_n)\big\|^2\Big] \leq\,\,& 2\mathbb{E}\Big[\big\|X(t_n)-\widetilde Y(t_n)\big\|^2\Big]+2\mathbb{E}\big[\|X(t_n)\|^2\big]\\
    \leq\,\,& 2\mathbb{E}\Big[\mathbb{E}\Big[\big\|X(t_n)-\widetilde Y(t_n)\big\|^2\Big| \mathcal{F}_{t_{n-1}}\Big]\Big]+2\mathbb{E}\bigg[\sup_{t_n\in[0,T]}\|X(t_n)\|^2\bigg]\\
    \leq\,\,& 2\mathbb{E}\Big[\mathbb{E}\Big[\big\|\widetilde E(t_n)\big\|^2\Big| \mathcal{F}_{t_{n-1}}\Big]\Big]+2\Cx. \numberthis \label{eq:iteration}
\end{align*}
\NOTE{Next, we repeatedly substitute \eqref{eq:err_before_sum} for decreasing values of $n$ into the RHS of \eqref{eq:iteration} until $n=0$. Then with tower property, Definition \ref{def:N}, \eqref{eq:defh} and \eqref{eq:E[Gamma_2]}, we have }

\begin{align*}
    \mathbb{E}\Big[\big\|\widetilde Y(t_n)\big\|^2\Big]  \leq\,\,& 2\mathbb{E}\Big[\big\|\widetilde E(t_{n-1})\big\|^2\Big]+2\Gamma_1(R)\mathbb{E}\Bigg[\int_{t_{n-1}}^{t_n}\mathbb{E}\Big[\big\|\widetilde E(r)\big\|^2\Big| \mathcal{F}_{t_{n-1}}\Big]dr\Bigg]\\
    &\qquad+2\mathbb{E}\Bigg[\overline \Gamma_2^{\{4(q+2)\}}(R)   h_{n} ^3\Bigg]+2\Cx\\
    \leq\,\,& 2\mathbb{E}\Big[\big\|\widetilde E(t_{0})\big\|^2\Big]+2\Gamma_1(R)\mathbb{E}\Bigg[\sum_{j=1}^{n}\int_{t_{j-1}}^{t_j}\mathbb{E}\Big[\big\|\widetilde E(r)\big\|^2\Big| \mathcal{F}_{t_{j-1}}\Big]dr\Bigg]\\
    &\qquad+2N^{(T)}_{\max}\Gamma_2(R)  \hmax ^3+2\Cx\\
    \leq\,\,& 2\Gamma_1(R)\mathbb{E}\Bigg[\int_{0}^{t_n}\mathbb{E}\Big[\big\|\widetilde E(r)\big\|^2\Big| \mathcal{F}_{t_{N^{(r)}-1}}\Big]dr\Bigg]\\
    &
    \qquad+2\left(\rho\,T+1 \right)\Gamma_2(R)  \hmax ^3+2\Cx. \numberthis \label{eq:prob_iterate}
\end{align*}
Since the integrand $\mathbb{E}\Big[\big\|\widetilde E(r)\big\|^2\Big| \mathcal{F}_{t_{N^{(r)}-1}}\Big]$ in the second term on the RHS of \eqref{eq:prob_iterate} is almost surely non-negative for all $r\in[0,T]$, we can replace the upper limit of integration with $T$. Using $\widetilde E(t_{0})=0$, \eqref{eq:defh}, the tower property of conditional expectation, and \eqref{eq:result_t} from Theorem \ref{thm:result}, we have 
\begin{multline}
    \mathbb{E}\Bigg[\int_{0}^{t_n}\mathbb{E}\Big[\big\|\widetilde E(r)\big\|^2\Big| \mathcal{F}_{t_{N^{(r)}-1}}\Big]dr\Bigg] \leq \int_{0}^{T}\mathbb{E}\Big[\big\|\widetilde E(r)\big\|^2\Big]dr\\ \leq T\max_{r\in[0,T]}\mathbb{E}\Big[\big\|\widetilde E(r)\big\|^2\Big] \leq T\,C^2(R,\rho,T)\hmax^2.   \label{eq:prob_Yn}
\end{multline}
By choosing $\hmax\leq 1/C(R,\rho,T)$, we substitute \eqref{eq:prob_Yn} into \eqref{eq:prob_iterate} and then \eqref{eq:prob} to get 
\begin{equation} 
    \mathbb{P}\big[h_{n+1}\,=\,\, h_{\min}\big] \,\leq\,\,
    \frac{2\Big(\Gamma_1(R)+\left(T+1 \right)\Gamma_2(R)  \hmax ^2+\Cx\Big)}{\rho^{2\kappa-1}}=:\frac{C_{\text{prob}}}{\rho^{2\kappa-1}},
\end{equation}
and the rest of the proof follows.
\end{proof}
\begin{appendices} 
\section{Proof of Lemma \ref{lem:levy bound} (\texttt{L\'evy Area})} \label{sec:appendix}
\label{sec:levybound}
\begin{proof} Set $\mathtt{i}^2=-1$.
Since the pair of Wiener processes $(W_i(r),W_j(r))^T$, $r\in [t_n, s]$, are mutually independent, by \cite[Eq. (1.3.5)]{levy1951wiener} the characteristic function of the L\'evy area \eqref{def:levy area} is given by
$\phi(\lambda)=(\cosh\left(\frac{1}{2}|s-t_n|\lambda\right))^{-1}.$
This was applied in the context of numerical methods for SDEs 
in \cite{malham2014efficient}. 
The Taylor expansion of the function $\cosh\left(\frac{1}{2}|s-t_n|\lambda\right)$ around $0$ gives
\begin{equation*}
\phi(\lambda)=\sum_{N=0}^{\infty}\frac{\textbf{E}_{2N}}{(2N)!}\left(\frac{1}{2}|s-t_n|\right)^{2N}\,\lambda^{2N},\quad \left|\frac{1}{2}|s-t_n|\lambda\right|<\frac{\pi}{2},
\end{equation*}
where $\textbf{E}_{2N}$ stands for the $2N^{\text{th}}$ Euler number, which may be expressed as
\begin{equation*}
\textbf{E}_{2N}=\mathtt{i}\sum_{b=1}^{2N+1}\sum_{j=0}^{b}\binom{j}{b}\frac{(-1)^j (b-2j)^{2N+1}}{2^b\,\mathtt{i}^b\,b},\quad N=0,1,2,3,\dots.
\end{equation*}
All odd Euler numbers are zero. The $k^{\text{th}}$ derivative of the characteristic function with respect to $\lambda$ is
\begin{equation*}
\phi(\lambda)_\lambda^{(k)}=\sum_{N=\lceil\frac{k}{2}\rceil}^{\infty}\left(\prod^{k-1}_{B=0}(2N-B)\right)\frac{\textbf{E}_{2N}}{(2N)!}\left(\frac{1}{2}|s-t_n|\right)^{2N}\,\lambda^{2N-k}.
\end{equation*}
As $\lambda\rightarrow 0$, since all terms vanish unless $k=2N$, we have
\begin{align*}
\lim_{\lambda\rightarrow 0}\phi(\lambda)_\lambda^{(k)}=\begin{cases}
\displaystyle\left(\prod^{k-1}_{B=0}(k-B)\right)\frac{\textbf{E}_{k}}{(k)!}\left(\frac{1}{2}|s-t_n|\right)^{k}, & k\,\text{ even};\\
\,\,0, & k\,\text{ odd}.
\end{cases}
\end{align*}
In the calculation of expectations, we make use of the mutual independence, conditional upon $\mathcal{F}_{t_n}$, of the pair of Brownian increments $(W_i(t),W_j(t))^T$.
Therefore, the $k^{\text{th}}$ conditional moment of $A_{ij}^{t_n,s}$ is 
\begin{equation*}
\mathbb{E}\Big[\big(A_{ij}^{t_n,s}\big)^k \Big|\mathcal{F}_{t_n}\Big]= L_k\,|s-t_n|^k,
\end{equation*}
where for all $a=1,2,3,\dots$
\begin{eqnarray*}
L_k&=& \left(\prod^{k-1}_{B=0}(k-B)\right)\frac{\textbf{E}_{k}}{(k)!}\left(-\frac{1}{2}\,\mathtt{i}\right)^{k}\\
&:=&\begin{cases}
 \left(\prod^{k-1}_{B=0}(k-B)\right)\frac{\textbf{E}_{k}}{(k)!}\left(\frac{1}{2}\right)^{k}, & \quad k=4a=4, 8, 12,\dots\\
 -\left(\prod^{k-1}_{B=0}(k-B)\right)\frac{\textbf{E}_{k}}{(k)!}\left(\frac{1}{2}\right)^{k}, & \quad k=4a-2=2, 6, 10 \dots \label{eq:Ib}\\
 0, & \quad k=2a-1=1, 3, 5,\dots
\end{cases}
\end{eqnarray*}
which is finite, as a finite product of finite factors. When $k$ is even, we have \begin{align*}
\mathbb{E}\Big[\big|A_{ij}^{t_n,s}\big|^k \Big|\mathcal{F}_{t_n}\Big]=
\mathbb{E}\Big[\big(A_{ij}^{t_n,s}\big)^k \Big|\mathcal{F}_{t_n}\Big]= L_k\,|s-t_n|^k, \quad a.s.
\end{align*} 
When $k$ is odd, i.e. $k=2c+1$ for all $c=0,1,2,\dots$, we have a.s.
\begin{align*}
\mathbb{E}\Big[\big|A_{ij}^{t_n,s}\big|^k \Big|\mathcal{F}_{t_n}\Big]&=\mathbb{E}\Big[\big|A_{ij}^{t_n,s}\big|^{2c+1} \Big|\mathcal{F}_{t_n}\Big]\\
&\leq\sqrt{
\mathbb{E}\Big[\big(A_{ij}^{t_n,s}\big)^{4c} \Big|\mathcal{F}_{t_n}\Big]\mathbb{E}\Big[\big(A_{ij}^{t_n,s}\big)^{2} \Big|\mathcal{F}_{t_n}\Big]}\\
&=\begin{cases}
\sqrt{L_2}\,|s-t_n|, &\quad c=0;\\
\displaystyle\sqrt{L_{4c}\cdot L_2}|s-t_n|^{2c+1}, &\quad c=1,2,3,\dots
\end{cases}\\
&=\begin{cases}
\displaystyle\sqrt{L_2}\,|s-t_n|, &\quad k=1;\\
\sqrt{L_{2k-2}\cdot L_2}|s-t_n|^{k}, &\quad k=3,5,7,\dots.
\end{cases}
\end{align*}
Therefore, in conclusion we have
\begin{align*}
\mathbb{E}\Big[\big|A_{ij}^{t_n,s}\big|^k \Big|\mathcal{F}_{t_n}\Big]\,\leq\,\, \Clevy{k}|s-t_n|^k,  \quad a.s.
\end{align*}
where
\begin{align*}
\Clevy{k}=\begin{cases}
\sqrt{L_2}, &\quad k=1;\\
\sqrt{L_{2k-2}\cdot L_2}, &\quad k=3,5,7,\dots; \numberthis \label{eq:widehat_Ib}\\
L_k, & \quad k=2,4,6,\dots
\end{cases}
\end{align*}
\end{proof}

\end{appendices}
\vspace{1cm}
The authors have no competing interests to declare that are relevant to the content of this article. 

\bibliographystyle{abbrv}
\bibliography{Manuscript}

\begin{thebibliography}{10}

\bibitem{beyn2017stochastic}
W.-J. Beyn, E.~Isaak, and R.~Kruse.
\newblock Stochastic {C}-stability and {B}-consistency of explicit and implicit
  {M}ilstein-type schemes.
\newblock {\em Journal of Scientific Computing}, 70(3):1042--1077, 2017.

\bibitem{Burrage2004}
P.~M. Burrage, R.~Herdiana, and K.~Burrage.
\newblock {Adaptive stepsize based on control theory for stochastic
  differential equations}.
\newblock {\em Journal of Computational and Applied Mathematics},
  171(1-2):317--336, 2004.

\bibitem{campbell2018adaptive}
S.~Campbell and G.~Lord.
\newblock Adaptive time-stepping for stochastic partial differential equations
  with non-{L}ipschitz drift.
\newblock {\em arXiv preprint arXiv:1812.09036}, 2018.

\bibitem{Dineen2011}
S.~Dineen.
\newblock {\em Probability Theory in Finance: a mathematical guide to the
  Black-Scholes Formula}.
\newblock Graduate studies in mathematics; v.70. American Mathematical Society,
  Universities Press, 2011.

\bibitem{fang2016adaptive}
W.~Fang and M.~B. Giles.
\newblock Adaptive {E}uler--{M}aruyama method for {SDE}s with non-globally
  {L}ipschitz drift.
\newblock In {\em International Conference on Monte Carlo and Quasi-Monte Carlo
  Methods in Scientific Computing}, pages 217--234. Springer, 2016.

\bibitem{fang2020adaptive}
W.~Fang and M.~B. Giles.
\newblock Adaptive {E}uler--{M}aruyama method for {SDE}s with nonglobally
  {L}ipschitz drift.
\newblock {\em The Annals of Applied Probability}, 30(2):526--560, 2020.

\bibitem{Gaines1997}
J.~Gaines and T.~Lyons.
\newblock {Variable step size control in the numerical solution of stochastic
  differential equations}.
\newblock {\em SIAM J. on Applied Math.}, 57(5):1455--1484, 1997.

\bibitem{gan2020tamed}
S.~Gan, Y.~He, and X.~Wang.
\newblock Tamed {R}unge-{K}utta methods for {SDE}s with super-linearly growing
  drift and diffusion coefficients.
\newblock {\em Applied Numerical Mathematics}, 152:379--402, 2020.

\bibitem{grasman2011stochastic}
J.~Grasman, H.~Salomons, and S.~Verhulst.
\newblock Stochastic modeling of length-dependent telomere shortening in
  {C}orvus monedula.
\newblock {\em Journal of Theoretical Biology}, 282(1):1--6, 2011.

\bibitem{Guo2018}
Q.~Guo, W.~Liu, X.~Mao, and R.~Yue.
\newblock The truncated {M}ilstein method for stochastic differential equations
  with commutative noise.
\newblock {\em Journal of Computational and Applied Mathematics}, 338:298 --
  310, 2018.

\bibitem{Hasminskii}
R.~Hasminskii.
\newblock {\em Stochastic Stability of Differential Equations}.
\newblock Sijthoff \& Noordhoff, 1980.

\bibitem{higham2002maple}
D.~J. Higham and P.~E. Kloeden.
\newblock Maple and {MATLAB} for stochastic differential equations in finance.
\newblock In {\em Programming Languages and Systems in Computational Economics
  and Finance}, pages 233--269. Springer, 2002.

\bibitem{HighamMaoSzpruch2013}
D.~J. Higham, X.~Mao, and L.~Szpruch.
\newblock Convergence, non-negativity and stability of a new {Milstein} scheme
  with applications to finance.
\newblock {\em Discrete and Continuous Dynamical Systems B}, pages 2083--2100,
  2013.

\bibitem{Hofmann2001}
N.~Hofmann, T.~M{\"u}ller-Gronbach, and K.~Ritter.
\newblock The optimal discretization of stochastic differential equations.
\newblock {\em Journal of Complexity}, 17(1):117 -- 153, 2001.

\bibitem{hutzenthaler2011strong}
M.~Hutzenthaler, A.~Jentzen, and P.~E. Kloeden.
\newblock Strong and weak divergence in finite time of {Euler}'s method for
  stochastic differential equations with non-globally {Lipschitz} continuous
  coefficients.
\newblock {\em Proceedings of the Royal Society of London A: Mathematical,
  Physical and Engineering Sciences}, 467(2130):1563--1576, 2011.

\bibitem{MR3325826}
S.~Ilie, K.~R. Jackson, and W.~H. Enright.
\newblock Adaptive time-stepping for the strong numerical solution of
  stochastic differential equations.
\newblock {\em Numer. Algorithms}, 68(4):791--812, 2015.

\bibitem{kelly2018adaptive}
C.~Kelly and G.~J. Lord.
\newblock Adaptive time-stepping strategies for nonlinear stochastic systems.
\newblock {\em IMA Journal of Numerical Analysis}, 38(3):1523--1549, 2018.

\bibitem{kelly2021adaptive}
C.~Kelly and G.~J. Lord.
\newblock Adaptive {E}uler methods for stochastic systems with non-globally
  {L}ipschitz coefficients.
\newblock {\em Numerical Algorithms}, pages 1--27, 2021.

\bibitem{kloeden1991numerical}
P.~Kloeden and E.~Platen.
\newblock Numerical methods for stochastic differential equations.
\newblock {\em Stochastic Hydrology and Hydraulics}, 5(2):172--172, 1991.

\bibitem{KumarSabanis2019}
C.~Kumar and S.~Sabanis.
\newblock On {M}ilstein approximations with varying coefficients: the case of
  super-linear diffusion coefficients.
\newblock {\em BIT Numerical Mathematics}, 59(4):929--968, 2019.

\bibitem{LMS2006}
H.~Lamba, J.~C. Mattingly, and A.~M. Stuart.
\newblock An adaptive {E}uler-{M}aruyama scheme for {SDE}s: convergence and
  stability.
\newblock {\em IMA J. Numer. Anal.}, 27(3):479--506, 2007.

\bibitem{levy1951wiener}
P.~L{\'e}vy.
\newblock {Wiener}'s random function, and other {Laplacian} random functions.
\newblock In {\em Proceedings of the Second Berkeley Symposium on Mathematical
  Statistics and Probability}. The Regents of the University of California,
  1951.

\bibitem{lord2014introduction}
G.~J. Lord, C.~E. Powell, and T.~Shardlow.
\newblock {\em An introduction to computational stochastic PDEs}, volume~50.
\newblock Cambridge University Press, 2014.

\bibitem{malham2014efficient}
S.~Malham and A.~Wiese.
\newblock Efficient almost-exact {L}{\'e}vy area sampling.
\newblock {\em Statistics \& Probability Letters}, 88:50--55, 2014.

\bibitem{mao2007SDEapp}
X.~Mao.
\newblock {\em Stochastic differential equations and applications}.
\newblock Woodhead Publishing, Cambridge, 2 edition, 2007.

\bibitem{mao2015truncated}
X.~Mao.
\newblock The truncated {E}uler--{M}aruyama method for stochastic differential
  equations.
\newblock {\em Journal of Computational and Applied Mathematics}, 290:370--384,
  2015.

\bibitem{mao2006stochastic}
X.~Mao and C.~Yuan.
\newblock {\em Stochastic differential equations with Markovian switching}.
\newblock Imperial college press, 2006.

\bibitem{Mauthner1998}
S.~Mauthner.
\newblock {Step size control in the numerical solution of stochastic
  differential equations}.
\newblock {\em J. of Comp. and Applied Math.}, 100(1):93--109, 1998.

\bibitem{Papoulis}
A.~Papoulis and S.~U. Pillai.
\newblock {\em Probability, random variables, and stochastic processes}.
\newblock McGraw-Hill, New York, 4 edition, 2002.

\bibitem{reisinger2020adaptive}
C.~Reisinger and W.~Stockinger.
\newblock An adaptive {E}uler-{M}aruyama scheme for {M}ckean-{V}lasov {SDE}s
  with super-linear growth and application to the mean-field
  {F}itz{H}ugh-{N}agumo model.
\newblock {\em arXiv preprint arXiv:2005.06034}, 2020.

\bibitem{shardlow2016pathwise}
T.~Shardlow and P.~Taylor.
\newblock On the pathwise approximation of stochastic differential equations.
\newblock {\em BIT Numerical Mathematics}, 56(3):1101--1129, 2016.

\bibitem{Shiryaev96}
A.~Shiryaev.
\newblock {\em Probability}.
\newblock Springer, Berlin, 2 edition, 1996.

\bibitem{tretyakov2013fundamental}
M.~V. Tretyakov and Z.~Zhang.
\newblock A fundamental mean-square convergence theorem for {SDE}s with locally
  {L}ipschitz coefficients and its applications.
\newblock {\em SIAM Journal on Numerical Analysis}, 51(6):3135--3162, 2013.

\bibitem{wang2013tamed}
X.~Wang and S.~Gan.
\newblock The tamed {Milstein} method for commutative stochastic differential
  equations with non-globally {Lipschitz} continuous coefficients.
\newblock {\em Journal of Difference Equations and Applications},
  19(3):466--490, 2013.

\bibitem{YaoGan18}
J.~Yao and S.~Gan.
\newblock Stability of the drift-implicit and double-implicit {Milstein}
  schemes for nonlinear {SDEs}.
\newblock {\em Applied Mathematics and Computation}, 339:294--301, 12 2018.

\end{thebibliography}

\end{document}